\documentclass[reqno,11pt]{amsart}
\usepackage[export]{adjustbox}
\usepackage{amsmath,amsthm,amsfonts,amssymb}
\usepackage{parskip}
\usepackage{fullpage}

\usepackage{hyperref}
\usepackage[dvipsnames]{xcolor}
\usepackage{comment}
\usepackage{graphicx,soul}
\usepackage{geometry}
\usepackage[shortlabels]{enumitem}
\usepackage{pgfplots}
\pgfplotsset{compat=1.18}
\usepackage{subcaption}
\hypersetup{
    colorlinks=true,
    citecolor=red,
    pdftitle    = {Diagonal Specht ideals and their varieties},
    pdfauthor   = {Anna Escofet, Cordian Riener, Hugues Verdure},
    pdfkeywords = {Specht ideals, diagonal action, symmetric group, radical ideals, Cohen--Macaulay rings},
    pdfcreator  = {\LaTeX}
}

\usepackage{float}
\usepackage{booktabs}
\usepackage{tikz}
\usetikzlibrary{shapes.geometric, arrows, positioning}
\tikzset{
  node/.style={draw=none, inner sep=0pt},
}
\usepackage{tikz-cd}
\usepackage{ytableau}
\usepackage[maxbibnames=3,giveninits=true,sorting=nyt,style=ieee,dashed=false,doi=false,url=false,isbn=false]{biblatex}
\newcommand{\C}{\mathbb C}
\newcommand{\K}{\mathbb K}
\newcommand{\N}{\mathbb N}

\newcommand{\Q}{\mathbb Q}
\newcommand{\fP}{\mathfrak{P}}
\newcommand{\fQ}{\mathfrak{Q}}
\newcommand{\balpha}{\boldsymbol{\alpha}}
\newcommand{\bbeta}{\boldsymbol{\beta}}
\newcommand{\bgamma}{\boldsymbol{\gamma}}
\newcommand{\bdelta}{\boldsymbol{\delta}}
\newcommand{\partn}{\mathcal{P}_n}
\renewcommand{\leq}{\leqslant}
\renewcommand{\geq}{\geqslant}

\DeclareMathOperator{\ccTab}{ccTab}

\DeclareMathOperator{\Col}{Col}
\DeclareMathOperator{\depth}{depth}
\DeclareMathOperator{\diag}{diag}
\DeclareMathOperator{\DR}{DR}
\DeclareMathOperator{\ev}{ev}
\DeclareMathOperator{\GL}{GL}
\DeclareMathOperator{\id}{id}
\DeclareMathOperator{\ini}{in}
\DeclareMathOperator{\len}{len}
\DeclareMathOperator{\mdeg}{mdeg}
\DeclareMathOperator{\Mon}{Mon}

\DeclareMathOperator{\rad}{rad}

\DeclareMathOperator{\sgn}{sgn}
\DeclareMathOperator{\sh}{sh}
\DeclareMathOperator{\SP}{SP}
\DeclareMathOperator{\spe}{sp}
\DeclareMathOperator{\Spec}{Spec}
\DeclareMathOperator{\Stab}{Stab}

\DeclareMathOperator{\Supp}{Supp}
\DeclareMathOperator{\Sym}{Sym}
\DeclareMathOperator{\SYT}{SYT}
\DeclareMathOperator{\Tab}{Tab}
\DeclareMathOperator{\Cont}{Cont}
\DeclareMathOperator{\Mdeg}{Mdeg}

\theoremstyle{plain}
\newtheorem{thm}{Theorem}[section]
\newtheorem{lemma}[thm]{Lemma}
\newtheorem{prop}[thm]{Proposition}
\newtheorem{corollary}[thm]{Corollary}
\newtheorem{conjecture}[thm]{Conjecture}
\newtheorem{mainthm}{Theorem}

\theoremstyle{definition}
\newtheorem{definition}[thm]{Definition}
\newtheorem{example}[thm]{Example}
\theoremstyle{remark}
\newtheorem{remark}[thm]{Remark}

\title{Diagonal Specht ideals and their varieties}
\author{Anna Escofet, Cordian Riener, Hugues Verdure}
\date{}

\begin{document}

\begin{abstract}
We study the diagonal Specht ideals $I_\lambda\subseteq\mathcal{R}_{m,n}:=\C[\mathbf{x}_1,\dots,\mathbf{x}_m]$, $\mathbf{x}_r=(x_{r,1},\dots,x_{r,n})$, generated by the $\lambda$-isotypic component of $\mathcal{R}_{m,n}$ for the diagonal action of $S_n$. We characterize the set of zeros of $I_\lambda$ as $V_\lambda=\bigcup_{\mu\not\trianglelefteq\lambda}H_\mu$ and prove that $\lambda\mapsto V_\lambda$ is an isomorphism of posets from $(\partn,\trianglelefteq)$ to $(\{V_\lambda:\lambda\in \partn\},\supseteq)$ for all $m$, while $\lambda\mapsto I_\lambda$ is an isomorphism of posets from $(\partn,\trianglelefteq)$ to $(\{I_\lambda:\lambda\in \partn\},\subseteq)$ if and only if $m=1$ or $n\leq3$. Unlike the case $m=1$, where Specht ideals are always radical, radicality in the diagonal setting depends on $\lambda$ and $m$. We prove radicality for hook partitions of length at most three by providing explicit Gr\"obner bases, and we develop three criteria for non-radicality via content, multidegree, and total degree, showing in particular that $I_\lambda$ is not radical for any non-hook partition $\lambda$ when $m\geq\len(\lambda)$. Using the $\GL_m(\C)$-action on $\mathcal{R}_{m,n}$, we show that $I_\lambda$ is radical for all $m$ if and only if it is radical for $m=n$. Finally, for $m\geq2$, we prove that $\mathcal{R}_{m,n}/I_\lambda$ is Cohen--Macaulay if and only if $\lambda=(n)$ or $\lambda=(n-1,1)$, and the same is true for $\mathcal{R}_{m,n}/\rad(I_\lambda)$.
\end{abstract}
\maketitle

\section{Introduction}

The representation theory of the symmetric group $S_n$ is classically well understood and has beautiful connections to combinatorics. The irreducible representations of $S_n$ are indexed by the partitions $\lambda\vdash n$, and a classical construction of Specht~\cite{specht1935irreduziblen} realizes them inside the polynomial ring
\[
\mathcal{R}_{1,n} := \C[x_1,\dots,x_n]
\]
via the so-called \emph{Specht polynomials}, which are products of Vandermonde determinants. More concretely, given $\lambda\vdash n$, a \emph{Young tableau} $T$ of shape $\lambda$ is a filling of the Young diagram of $\lambda$ with the elements of $[n]:=\{1,\dots,n\}$, each used exactly once. Writing $C_1,\dots,C_{\lambda_1}$ for the sets of entries in the columns of $T$, the Specht polynomial associated to $T$ is
\[
\spe_T := \prod_{j=1}^{\lambda_1} \Delta_{C_j},
\qquad\text{where}\qquad
\Delta_{C} := \prod_{\substack{a<b\\ a,b\in C}} (x_a - x_b).
\]
It is homogeneous of degree $\sum_i (i-1)\lambda_i$, and $S_n$ permutes the Specht polynomials of a fixed shape by acting on the variables. Their linear span is an $S_n$-submodule of $\mathcal{R}_{1,n}$, the \emph{Specht module} $S^\lambda$; it is irreducible, and the Specht polynomials of the standard Young tableaux form a basis of it. As $\lambda$ ranges over all partitions of $n$, the modules $S^\lambda$ form a complete set of pairwise non-isomorphic irreducible representations of $S_n$.

The ideal $I_\lambda\subseteq\mathcal{R}_{1,n}$ generated by the Specht polynomials of shape $\lambda$ is the \emph{Specht ideal} of shape $\lambda$. Specht ideals provide a concrete link between the representation theory of $S_n$, symmetric polynomial systems, and the stratification of $\C^n$ with respect to stabilizer subgroups, and they were studied from this point of view by Moustrou, Riener, and Verdure~\cite{moustrou2021symmetric}. Given a partition  $\mu=(\mu_1,\ldots,\mu_k)\vdash n$, we consider the set of  points $H_\mu\subseteq\C^n$ whose stabilizer is conjugate to $S_{\mu_1}\times\cdots\times S_{\mu_k}$. Note that this just means that the points in $H_\mu$ have exactly $k$ distinct coordinates and that the multiplicity of occurrence of the distinct coordinates corresponds to $(\mu_1,\ldots,\mu_k)$. It is shown in~\cite{moustrou2021symmetric} that the zero set of $I_\lambda$ is
\[
V_\lambda = \bigcup_{\mu\not\trianglelefteq\lambda} H_\mu,
\]
and that the maps $\lambda\mapsto V_\lambda$ and $\lambda\mapsto I_\lambda$ identify the dominance order on partitions of $n$ with reverse inclusion of the varieties and with inclusion of the ideals, respectively. Furthermore, in characteristic zero, Specht ideals are radical and have strong Gr\"obner basis properties; see Haiman--Woo~\cite{Woo2005} and Murai--Ohsugi--Yanagawa~\cite{Murai2022}. We further note that the notion of Specht ideals has recently been studied beyond the symmetric groups: Debus, Moustrou, Riener, and Verdure developed an analogue for the hyperoctahedral groups $B_n$~\cite{debus2023poset}, and Debus and Gottwald completed the analogous study for all infinite families of essential real reflection groups~\cite{debus2025posetsspechtidealsessential}. These results in particular  highlight that the interplay between representation-theoretic generators, orbit-type stratifications, and posets of invariant ideals extends naturally beyond type~$A$. In the present paper we take a different direction: instead of changing the acting reflection group, we keep $S_n$ and pass from one set of variables to several sets, equipped with the diagonal action. The guiding question is how much of the classical picture survives this passage. In short, the answer is: set-theoretically all of it, scheme-theoretically very little. Before we state our main results we briefly  set the stage by defining the objects:

\paragraph{\bf Diagonal Specht ideals}
As explained above, in one set of variables the Specht polynomials, which realize the irreducible representations, give a natural definition of the Specht ideals. In the diagonal setting we proceed differently.  
Let
\[
\mathcal{R}_{m,n} := \C[\mathbf{x}_1,\dots,\mathbf{x}_m],
\qquad
\mathbf{x}_r = (x_{r,1},\dots,x_{r,n}).
\]
The symmetric group $S_n$ acts diagonally, by simultaneously permuting the second indices, and $\GL_m:=\GL_m(\C)$ acts by linear transformations among the $m$ rows $\mathbf{x}_1,\dots,\mathbf{x}_m$ of variables; the two actions commute. Geometrically, we regard a point of $\C^{mn}$ as an $m\times n$ matrix $\mathbf{u}=(u_{r,i})$, that is, as an ordered configuration of $n$ column vectors $\mathbf{u}_{\bullet,1},\dots,\mathbf{u}_{\bullet,n}\in\C^m$; then $S_n$ permutes the columns and $\GL_m$ acts on the rows. Throughout, we work over $\C$; the arguments only use that the ground field is algebraically closed of characteristic zero, and all results remain valid over any such field (see also Remark~\ref{rem:char0}).
For a partition $\lambda\vdash n$, let $\mathcal{R}_{m,n,\lambda}$ denote the $\lambda$-isotypic component of $\mathcal{R}_{m,n}$ for the diagonal action of $S_n$. For $\lambda=(n)$ this is the ring of invariants; for $\lambda\neq(n)$ it is not a ring, but only a module over the invariants. We define the \emph{diagonal Specht ideal} of shape $\lambda$ to be
\[
I_\lambda := \langle \mathcal{R}_{m,n,\lambda}\rangle \subseteq \mathcal{R}_{m,n}.
\]
For $m=1$ this agrees with the usual Specht ideal (Proposition~\ref{prop:defequiv}). Defining the ideal through isotypic components is particularly natural in the diagonal setting: there is no single preferred family of Specht polynomials in several sets of variables, whereas the $S_n$-isotypic decomposition is intrinsic and compatible with the commuting $\GL_m$-action.

One instance of this construction is classical. For $m=2$ and $\lambda=(1^n)$, the ideal \[I_{(1^n)}\subseteq\C[x_1,\dots,x_n,y_1,\dots,y_n]\] is generated by the $S_n$-alternating polynomials, and Haiman proved that it coincides with the ideal \[\bigcap_{1\leq i<j\leq n}\langle x_i-x_j,\,y_i-y_j\rangle\] of the big diagonal $D_{2,n}=\bigcup_{i<j}\{(x_i,y_i)=(x_j,y_j)\}\subseteq(\C^2)^n$; this is the case $d=1$ of \cite[Corollary~3.8.3]{Haiman2001}. In particular, this diagonal Specht ideal is radical. Our approach differs in that the ideals are defined representation-theoretically, and in that we treat every partition $\lambda\vdash n$ and every $m\geq1$.

\paragraph{\bf Main results}
Our first result is that the set-theoretic geometry of Specht ideals is unchanged in the diagonal setting. For $\mu\vdash n$, let $H_\mu\subseteq\C^{mn}$ be the locus of matrices whose $n$ columns take exactly $\len(\mu)$ distinct values in $\C^m$, with multiplicities given by the parts of $\mu$; for $m=1$ these are the strata introduced above. For a set partition $\Pi$ of $[n]$, let $W_\Pi\subseteq\C^{mn}$ be the linear subspace on which columns indexed by the same block of $\Pi$ coincide.

\begin{mainthm}[{Theorem~\ref{thm:charzeros}, Corollaries~\ref{cor:linearspaces} and~\ref{cor:posetvar}}]\label{thmA}
The zero set of $I_\lambda\subseteq\mathcal{R}_{m,n}$ is
\[
V_\lambda=\bigcup_{\mu\not\trianglelefteq\lambda}H_\mu.
\]
Consequently, for every $m\geq1$, the map $\lambda\longmapsto V_\lambda$ is an isomorphism of posets from $(\partn,\trianglelefteq)$ to $(\{V_\lambda:\lambda\in\partn\},\supseteq)$, and the irreducible components of $V_\lambda$ are the subspaces $W_\Pi$ whose shape is minimal, with respect to refinement, among the partitions not dominated by $\lambda$.
\end{mainthm}

\noindent In particular, $\rad(I_\lambda)$ is an intersection of prime ideals generated by linear forms. We also compute $\dim V_\lambda=m(n-\lambda_1)$ for $\lambda\neq(n)$, and show that $V_\lambda$ is equidimensional exactly when $\lambda$ is a hook or $\lambda_1=\dots=\lambda_{\len(\lambda)-1}$ (Propositions~\ref{prop:dimension} and~\ref{prop:equidim}).

\medskip
The scheme-theoretic behavior is substantially different for $m\geq2$: diagonal Specht ideals need not be radical, and radicality depends on both $\lambda$ and $m$. The dependence on $m$ is controlled by the $\GL_m$-action. Radicality can only be lost as $m$ grows (Proposition~\ref{prop:bound}), and it suffices to check the case $m=n$.

\begin{mainthm}[{Theorem~\ref{thm:radicalinRnn}}]\label{thmB}
If $I_\lambda\subseteq\mathcal{R}_{n,n}$ is radical, then $I_\lambda\subseteq\mathcal{R}_{m,n}$ is radical for every $m\geq1$.
\end{mainthm}

\noindent Consequently, every $\lambda\vdash n$ has a well-defined \emph{radicality threshold}
\[
m(\lambda) := \min\{\,m\geq1 : I_\lambda\subseteq\mathcal{R}_{m,n}\ \text{is not radical}\,\}\ \in\ \{2,\dots,n\}\cup\{\infty\}
\]
(Definition~\ref{def:threshold}), where $\min\emptyset:=\infty$: the ideal $I_\lambda\subseteq\mathcal{R}_{m,n}$ is radical if and only if $m<m(\lambda)$, and $m(\lambda)=\infty$ as soon as $m(\lambda)>n$. Our remaining results on radicality determine or bound this threshold.

For the smallest hooks the threshold is infinite. Two cases are elementary: $I_{(n)}=\mathcal{R}_{m,n}$, since the trivial isotypic component contains the constants, and $V_{(n-1,1)}$ is the small diagonal $\{\mathbf{u}_{\bullet,1}=\dots=\mathbf{u}_{\bullet,n}\}$, whose prime ideal is generated by the linear forms $x_{r,i}-x_{r,j}\in\mathcal{R}_{m,n,(n-1,1)}$, so that $I_{(n-1,1)}$ is this prime ideal. The first substantial case is $\lambda=(n-2,1,1)$, where $V_\lambda$ is the union of the $2^{n-1}-1$ linear subspaces on which the columns take at most two distinct values.

\begin{mainthm}[{Theorems~\ref{thm:Ia1} and~\ref{thm:a11}}]\label{thmC}
Let $\lambda\vdash n$ be a hook partition with $\len(\lambda)\leq 3$. Then $I_\lambda\subseteq\mathcal{R}_{m,n}$ is radical for every $m\geq1$, that is, $m(\lambda)=\infty$. For $\lambda=(n-2,1,1)$ we give an explicit Gr\"obner basis of $I_\lambda$.
\end{mainthm}

\noindent We conjecture that the threshold is infinite for every hook partition (Conjecture~\ref{conj:hook}). Besides Theorem~\ref{thmC}, this is known for $\lambda=(1^n)$ and $m=2$ by Haiman's result quoted above, and for $\lambda=(1^4)$, $\lambda=(2,1^3)$ and all $m$ by a computation in $\mathcal{R}_{4,4}$ and $\mathcal{R}_{5,5}$, respectively, combined with Theorem~\ref{thmB} (Remark~\ref{rmk:haiman}).

On the negative side, we develop three criteria for non-radicality, based on contents, on multidegrees, and on total degrees (Proposition~\ref{prop:contentcriterion}, Corollaries~\ref{cor:mdegcriterion} and~\ref{cor:totaldegreecriterion}). They rest on a single mechanism. For $\mu\triangleleft\lambda$ we have $I_\mu\subseteq I(V_\mu)\subseteq I(V_\lambda)$ by Theorem~\ref{thmA}, so every polynomial of $S_n$-type $\mu$ lies in the radical of $I_\lambda$; if such a polynomial occurs in a content, a multidegree, or a total degree that cannot be reached from the $\lambda$-isotypic component by multiplication with monomials, then it does not lie in $I_\lambda$. The three criteria are successively coarser and successively easier to apply. The last one rests on an explicit formula for the minimal degree of $I_\lambda$ (Proposition~\ref{prop:degreegen}) and yields a purely combinatorial condition on $\lambda$ and $m$ for non-radicality (Theorem~\ref{thm:nonradical}), whose simplest instance is the following.

\begin{mainthm}[{Corollary~\ref{cor:nonhooks}}]\label{thmD}
Let $\lambda\vdash n$ be a partition that is not a hook. Then $m(\lambda)\leq\len(\lambda)$; that is, $I_\lambda\subseteq\mathcal{R}_{m,n}$ is not radical for every $m\geq\len(\lambda)$.
\end{mainthm}

\noindent This bound is attained for $(2,2,1)$ and $(3,2,1)$, whose ideals are still radical for $m=2$ (Section~\ref{subsec:summary}), but it is not attained in general: for $(2,2,2)$ the threshold is $2$, and Theorem~\ref{thm:nonradical} detects non-radicality below $\len(\lambda)$ for instance for $(3,3,3,3,2,1)$ and $m=2$. Table~\ref{tab:threshold} collects what we know about $m(\lambda)$ for $n\leq7$; the smallest open cases are $(1^5)$ for $m=3,4,5$ and $(2,2,1,1)$ for $m=2,3$.

\medskip
The containment relations between the ideals also differ from the ordinary case: dominance always governs the diagonal Specht varieties, but for $m\geq2$ it does not, in general, govern inclusion among the ideals themselves.

\begin{mainthm}[{Corollary~\ref{cor:posetideals}}]\label{thmE}
The map $\lambda\longmapsto I_\lambda$ is an isomorphism of posets from $(\partn,\trianglelefteq)$ to $(\{I_\lambda:\lambda\in\partn\},\subseteq)$ if and only if $m=1$ or $n\leq3$.
\end{mainthm}

\noindent The smallest instance occurs for $n=4$ and $m=2$: we have $(2,1,1)\trianglelefteq(2,2)$, but $I_{(2,1,1)}\not\subseteq I_{(2,2)}$, since both ideals start in degree two, where each coincides with its isotypic component, and distinct isotypic components intersect trivially. The failure is tied to non-radicality: the map $\lambda\mapsto I_\lambda$ always reflects dominance, and $I_\mu\subseteq I_\lambda$ holds whenever $\mu\trianglelefteq\lambda$ and $I_\lambda$ is radical (Remark~\ref{rmk:inclusionradicality}).

\medskip
Finally, we study Cohen--Macaulayness. For $m=1$, Yanagawa proved that, in characteristic zero, $\mathcal{R}_{1,n}/I_\lambda$ is Cohen--Macaulay precisely when $\lambda$ is of one of the forms $(a,1^{n-a})$, $(a,b)$ or $(a,a,1)$~\cite{Yanagawa2021}; in particular, all hooks yield Cohen--Macaulay quotients. In the diagonal setting the answer is far more rigid.

\begin{mainthm}[{Theorem~\ref{thm:CM}}]\label{thmF}
Let $\lambda\vdash n$ and $m\geq2$. Then $\mathcal{R}_{m,n}/I_\lambda$ is Cohen--Macaulay if and only if $\lambda=(n)$ or $\lambda=(n-1,1)$. The same holds for $\mathcal{R}_{m,n}/\rad(I_\lambda)$.
\end{mainthm}

\noindent In other words, for $m\geq2$ the only Cohen--Macaulay diagonal Specht ideals are the two trivial ones, the unit ideal $I_{(n)}$ and the prime ideal $I_{(n-1,1)}$ of the small diagonal. Already the radical ideal $I_{(n-2,1,1)}$ of Theorem~\ref{thmC} fails to be Cohen--Macaulay, and so does the radical of every other diagonal Specht ideal.

\paragraph{\bf Methods}
Theorem~\ref{thmA} is proved by reduction to the case $m=1$: the ideal $I_\lambda$ is generated by the images $\varepsilon_T(p)$ of monomials $p$ under the Young symmetrizers $\varepsilon_T$ (Proposition~\ref{prop:genIlambda}), and a generic linear combination of the rows of variables transports the classical characterization from $\mathcal{R}_{1,n}$ to $\mathcal{R}_{m,n}$. Theorem~\ref{thmB} relies on the Cauchy decomposition of $\mathcal{R}_{m,n}$ as a $\GL_m\times\GL_n$-module: a $\GL_m$-highest weight vector involves at most $n$ rows of variables, so a $\GL_m$-stable ideal is determined by its intersection with $\mathcal{R}_{n,n}$. Theorem~\ref{thmC} is proved by exhibiting a Gr\"obner basis of $I(V_{(n-2,1,1)})$ with respect to a block monomial order adapted to the columns (Definition~\ref{def:monomialordergen}), by induction on $n$ from a base case $n=3$ that is verified by computer (Lemma~\ref{lem:computationm=3}). Theorem~\ref{thmF} follows from Hartshorne's connectedness theorem for the punctured spectrum of a Cohen--Macaulay local ring, applied to a suitable localization; the argument only uses the minimal primes and therefore applies to $I_\lambda$ and $\rad(I_\lambda)$ alike. All computer calculations were carried out in \texttt{SageMath} 10.6; the code is available at \url{https://github.com/uggtos/diagonal_specht}.

\paragraph{\bf Related work}
The diagonal action of $S_n$ on $\mathcal{R}_{2,n}=\C[x_1,\dots,x_n,y_1,\dots,y_n]$ has attracted considerable attention. The diagonal coinvariant ring $\DR_n$ and the Garsia--Haiman modules $\DR_\mu$ are quotients of $\mathcal{R}_{2,n}$; Haiman proved that $\dim_\C\DR_n=(n+1)^{n-1}$ and that $\dim_\C\DR_\mu=n!$ for every $\mu\vdash n$, the latter being the key step in the proof of the Macdonald positivity conjecture~\cite{Haiman2001,Haiman2002}. Carlsson and Oblomkov constructed a monomial basis of $\DR_n$~\cite{CarlssonOblomkov2018}, and Gillespie developed higher Specht polynomials for the diagonal action and constructed a higher Specht basis of $\DR_\mu$ for hook partitions $\mu$~\cite{Gillespie2024}; the construction of such bases for arbitrary shapes, and for $\DR_n$ itself, remains open. These results concern quotients of $\mathcal{R}_{2,n}$ by the ideal of positive-degree invariants or by the annihilator of a single alternant, and hence a different family of ideals; the point of contact with the present paper is the alternant ideal $I_{(1^n)}$ discussed above.

From the geometric point of view introduced above, conditions requiring specified columns of $\mathbf{u}$ to coincide define linear subspaces of $\C^{mn}$, and their unions form diagonal subspace arrangements. The union of the pairwise diagonals,
\[
D_{m,n} := \bigcup_{1\leq i<j\leq n}\{\mathbf{u}_{\bullet,i}=\mathbf{u}_{\bullet,j}\}\subseteq(\C^m)^n,
\]
is the \emph{big diagonal}, and its complement is the ordered configuration space of $n$ distinct points in $\C^m$. For $m=1$, the collision subspaces are the elements of the intersection lattice of the braid arrangement, which is naturally identified with the lattice of set partitions of $[n]$: a set partition $\Pi$ corresponds to the subspace on which $x_i=x_j$ whenever $i\sim_\Pi j$. The same set partitions index the irreducible components of the diagonal Specht varieties (Corollary~\ref{cor:linearspaces}).

A close precursor of our work is the study by Li and Li of the ideal $J(k,n)$ of polynomials in $\C[x_1,\dots,x_n]$ that vanish whenever $k$ of the variables are set equal, that is, the vanishing ideal of the $k$-equal arrangement~\cite{LiLi}. They proved that $J(k,n)$ is generated by products of Vandermonde factors indexed by set partitions of $[n]$ into $k-1$ blocks, motivated by the observation that membership of the graph polynomial $f_G=\prod_{\{i,j\}\in E(G)}(x_i-x_j)$ in $J(k,n)$ detects whether the independence number of $G$ is smaller than $k$. Lov\'asz reformulated this result in terms of stable sets and discussed a dual statement concerning graph colouring~\cite{Lovasz}. Li and Li also showed that the radical ideal of a union of linear subspaces need not be generated by products of linear forms~\cite{LiLigenerators}. This distinction is relevant here: our varieties are finite unions of linear collision subspaces, but the diagonal Specht ideals need not be radical, and their generators arise from $S_n$-isotypic components rather than from products of linear forms.

Bj\"orner, Peeva, and Sidman placed these questions in the framework of subspace arrangements embedded in hyperplane arrangements and introduced \emph{blocker duality}~\cite{BjornerPeevaSidman}. The blocker dual of the $k$-equal arrangement is the arrangement of all subspaces indexed by set partitions with $k-1$ blocks, which is precisely the ordinary Specht variety of the hook $(n-k+1,1^{k-1})$. This gives an arrangement-theoretic interpretation of hook Specht ideals in one set of variables. In several sets of variables the same collision subspaces occur, but the diagonal Specht ideals are, in general, neither blocker ideals nor generated by products of linear forms. Subspace arrangements have also been studied from homological and computational viewpoints: Derksen and Sidman proved bounds for the Castelnuovo--Mumford regularity of subspace arrangements~\cite{DerksenSidman2002}, and the topology and combinatorics of $k$-equal arrangements are governed by restricted partition lattices; see Bj\"orner and Welker~\cite{BjornerWelker}.

\paragraph{\bf Outline}
The article is structured as follows. 
Section~\ref{sec:preliminaries} recalls the necessary background on partitions, tableaux and Specht polynomials, and on the representation theory of $\GL_m\times S_n$ on $\mathcal{R}_{m,n}$ and its isotypic decomposition. In Section~\ref{sec:spechtideals} we define diagonal Specht ideals and show that the definition extends the classical one. Section~\ref{sec:varieties} describes their zero sets, irreducible components, dimensions and equidimensionality. Section~\ref{sec:radicality} is devoted to radicality: we first show that radicality is controlled by the case $m=n$, then compute Gr\"obner bases of the ideals $I_{(n-1,1)}$ and $I_{(n-2,1,1)}$ and deduce that they are radical, and finally establish general non-radicality criteria and examples. Section~\ref{sec:cohen-macaulayness} studies Cohen--Macaulayness of the diagonal Specht ideals and of their radicals, and Section~\ref{sec:open} collects open questions.

\section{Preliminaries}\label{sec:preliminaries}
In this section we collect preliminary material, such as definitions and statements, helpful later. 
\subsection{Partitions and set partitions}

\begin{definition}\label{def:partitions}
A \emph{partition} of $n\in\N$ is a finite sequence $\lambda=(\lambda_1,\dots,\lambda_\ell)$ of positive integers with $\lambda_1\geq\cdots\geq\lambda_\ell$ and $\lambda_1+\cdots+\lambda_\ell=n$; we write $\lambda\vdash n$. The number $\ell$ is the \emph{length} of $\lambda$, denoted by $\len(\lambda)$, and by convention $\lambda_k=0$ for $k>\ell$. We denote by $\partn$ the set of all partitions of $n$. A partition $\lambda\vdash n$ is a \emph{hook partition} if it is of the form $(a,1^{n-a})$ for some $1\leq a\leq n$, and a \emph{non-hook partition} otherwise.
\end{definition}

\begin{definition}\label{def:orders}
Let $\lambda,\mu\in\partn$.
\begin{enumerate}
    \item We say that $\lambda$ \emph{dominates} $\mu$, and write $\mu\trianglelefteq\lambda$, if $\sum_{i=1}^k\mu_i\leq\sum_{i=1}^k\lambda_i$ for all $k\geq1$. We write $\mu\triangleleft\lambda$ if $\mu\trianglelefteq\lambda$ and $\mu\neq\lambda$.
    \item We say that $\mu$ \emph{refines} $\lambda$, and write $\mu\prec\lambda$, if there is a surjection $\pi:\{1,\dots,\len(\mu)\}\to\{1,\dots,\len(\lambda)\}$ such that $\lambda_j=\sum_{i\in\pi^{-1}(j)}\mu_i$ for all $j=1,\dots,\len(\lambda)$; in other words, if $\lambda$ is obtained from $\mu$ by merging parts. In particular, $\lambda\prec\lambda$.
\end{enumerate}
Both relations are partial orders on $\partn$, called the \emph{dominance order} and the \emph{refinement order}, respectively.
\end{definition}

\begin{remark}\label{rem:refinementdominance}
Merging parts can only increase partial sums, so $\mu\prec\lambda$ implies $\mu\trianglelefteq\lambda$. The converse fails: $(3,3)\trianglelefteq(4,2)$, but $(3,3)\not\prec(4,2)$.
\end{remark}

\begin{example}\label{ex:dominance6}
The dominance order is total for $n\leq5$; for $n=6$ its Hasse diagram is
\begin{center}
\ytableausetup{boxsize=0.15cm}
\begin{tikzpicture}[node distance=2cm and 1.5cm, every node/.style={node}]
  % Nodes
  \node (111111) {\ydiagram{1,1,1,1,1,1}};
  \node (21111) [right of=111111] {\ydiagram{2,1,1,1,1}};
  \node (2211) [right of=21111] {\ydiagram{2,2,1,1}};
  \node (222) [below right of=2211] {\ydiagram{2,2,2}};
  \node (3111) [above right of=2211] {\ydiagram{3,1,1,1}};
  \node (321) [below right of=3111] {\ydiagram{3,2,1}};
  \node (33) [below right of=321] {\ydiagram{3,3}};
  \node (411) [above right of=321] {\ydiagram{4,1,1}};
  \node (42) [below right of=411] {\ydiagram{4,2}};
  \node (51) [right of=42] {\ydiagram{5,1}};
  \node (6) [right of=51] {\ydiagram{6}};
  % Edges
  \draw[-, shorten <=5pt, shorten >=5pt] (111111) -- (21111);
  \draw[-, shorten <=5pt, shorten >=5pt] (21111) -- (2211);
  \draw[-, shorten <=5pt, shorten >=5pt] (2211) -- (222);
  \draw[-, shorten <=5pt, shorten >=5pt] (2211) -- (3111);
  \draw[-, shorten <=5pt, shorten >=5pt] (222) -- (321);
  \draw[-, shorten <=5pt, shorten >=5pt] (3111) -- (321);
  \draw[-, shorten <=5pt, shorten >=5pt] (321) -- (33);
  \draw[-, shorten <=5pt, shorten >=5pt] (321) -- (411);
  \draw[-, shorten <=5pt, shorten >=5pt] (33) -- (42);
  \draw[-, shorten <=5pt, shorten >=5pt] (411) -- (42);
  \draw[-, shorten <=5pt, shorten >=5pt] (42) -- (51);
  \draw[-, shorten <=5pt, shorten >=5pt] (51) -- (6);
\end{tikzpicture}
\ytableausetup{boxsize=0.4cm}
\end{center}
with $(2,2,2)$ and $(3,1,1,1)$, as well as $(3,3)$ and $(4,1,1)$, incomparable. In the refinement order, $(2,2,2)\prec(4,2)$ since $4=2+2$, $(4,1,1)\prec(4,2)$ since $2=1+1$, and $(3,2,1)\prec(3,3)$ since $3=2+1$, whereas $(3,3)$, $(4,2)$ and $(5,1)$ are pairwise incomparable.
\end{example}

\begin{definition}\label{def:posetiso}
A map $\varphi:(A,\leq_A)\longrightarrow(B,\leq_B)$ between partially ordered sets is an \emph{isomorphism of posets} if it is bijective and $x\leq_A y$ if and only if $\varphi(x)\leq_B\varphi(y)$ for all $x,y\in A$.
\end{definition}

\begin{definition}\label{def:setpartitions}
Let $[n]:=\{1,\dots,n\}$. A \emph{set partition} of $[n]$ into $k$ subsets is a set $\Pi=\{\Pi_1,\dots,\Pi_k\}$ of pairwise disjoint nonempty subsets of $[n]$, called the \emph{blocks} of $\Pi$, with $\Pi_1\sqcup\dots\sqcup\Pi_k=[n]$; we denote by $\SP(n,k)$ the set of such set partitions. For $i,j\in[n]$ we write $i\sim_\Pi j$ if $i$ and $j$ lie in the same block of $\Pi$. The \emph{shape} $\sh(\Pi)\vdash n$ of $\Pi$ is the partition obtained by arranging the block sizes $|\Pi_1|,\dots,|\Pi_k|$ in weakly decreasing order, and $\SP_\mu$ denotes the set of set partitions of shape $\mu$. A set partition $\Theta$ is \emph{coarser} than $\Pi$, written $\Pi\prec\Theta$, if $i\sim_\Pi j$ implies $i\sim_\Theta j$.
\end{definition}

\begin{remark}\label{rem:coarserrefines}
If $\Pi\prec\Theta$, then every block of $\Theta$ is a union of blocks of $\Pi$, hence $\sh(\Pi)\prec\sh(\Theta)$.
\end{remark}

\subsection{Tableaux and Young symmetrizers}
 
\begin{definition}\label{def:tableaux}
For a partition $\lambda\vdash n$ a  \emph{Young diagram} of shape $\lambda$ is the left-justified array of boxes with $\lambda_i$ boxes in the $i$-th row from the top. A \emph{tableau} of shape $\lambda$ is a filling of these boxes with the numbers $1,\dots,n$, each appearing exactly once; it is \emph{standard} if its entries increase along every row and down every column. We write $\Tab(\lambda)$ and $\SYT(\lambda)$ for the sets of tableaux and of standard tableaux of shape $\lambda$.
\end{definition}
 
For example, for $\lambda=(3,2,1)$,
\begin{equation*}
A=\begin{ytableau}1&2&3\\4&5\\6\end{ytableau}\in\SYT(\lambda),
\qquad
T=\begin{ytableau}2&5&1\\6&3\\4\end{ytableau}\in\Tab(\lambda)\setminus\SYT(\lambda).
\end{equation*}
The tableau $A$, filled with $1,\dots,n$ row by row from left to right and from top to bottom, is the \emph{superstandard} tableau of shape $\lambda$.
 
\subsubsection*{Notation}
For $T\in\Tab(\lambda)$ and a box $b$ of the Young diagram, $T(b)$ denotes the entry of $T$ in $b$. Conversely, for $1\leq i\leq n$, $b_T(i)$ denotes the box of $T$ containing $i$, and $r_T(i)$ the row of that box; in the example above, $r_T(3)=2$ and $r_T(4)=3$.
 
\begin{definition}\label{def:cocharge}
The \emph{reading word} of $S\in\Tab(\lambda)$ is the word $w(S)=w_1\cdots w_n$ obtained by reading the entries of $S$ row by row from bottom to top, and from left to right within each row. Following~\cite{Gillespie2024}, the \emph{cocharge labels} of a word $w=w_1\cdots w_n$ are assigned as follows: the letter $1$ receives the label $0$, and for $i=2,\dots,n$ in turn one searches for $i$ cyclically to the left of $i-1$; the letter $i$ receives the label of $i-1$ increased by one, unless the search wrapped around the word, in which case it receives the same label as $i-1$. The \emph{cocharge tableau} $\ccTab(S)$ is the filling of the Young diagram of $\lambda$ in which the box occupied by $i$ in $S$ contains the cocharge label of $i$ in $w(S)$.
\end{definition}
 
\begin{example}\label{ex:cocharge}
For $\lambda=(3,2,1)$,
\begin{equation*}
S=\begin{ytableau}1&2&3\\4&6\\5\end{ytableau},
\qquad
w(S)=5_24_16_21_02_03_0,
\qquad
\ccTab(S)=\begin{ytableau}0&0&0\\1&2\\2\end{ytableau},
\end{equation*}
where the subscripts are the cocharge labels.
\end{example}
 
\begin{definition}\label{def:youngsymmetrizer}
Let $T\in\Tab(\lambda)$. The \emph{row stabilizer} $R(T)$ and the \emph{column stabilizer} $C(T)$ of $T$ are the subgroups of $S_n$ consisting of the permutations that preserve the set of entries of every row, respectively of every column, of $T$. The \emph{column antisymmetrizer}, the \emph{row symmetrizer} and the \emph{Young symmetrizer} of $T$ are the elements
\begin{equation*}
\alpha(T):=\sum_{\tau\in C(T)}\sgn(\tau)\,\tau, \qquad \beta(T):=\sum_{\sigma\in R(T)}\sigma, \qquad \varepsilon_T:=\alpha(T)\beta(T)
\end{equation*}
of the group algebra $\C[S_n]$.
\end{definition}

\subsection{The polynomial ring \texorpdfstring{$\mathcal{R}_{m,n}$}{R\_\{m,n\}} and its symmetries}

\begin{definition}\label{def:Rmn}
For $m,n\geq1$ let
\begin{equation*}
\mathcal{R}_{m,n}:=\C[\mathbf{x}_1,\dots,\mathbf{x}_m],\qquad \mathbf{x}_r=(x_{r,1},\dots,x_{r,n}),
\end{equation*}
be the polynomial ring in the $mn$ variables $x_{r,i}$, which we arrange in an $m\times n$ matrix with rows $\mathbf{x}_1,\dots,\mathbf{x}_m$. The symmetric group $S_n$ acts on $\mathcal{R}_{m,n}$ by permuting the columns of this matrix,
\begin{equation}\label{eq:actiononpolynomialsgen}
\sigma\cdot p(x_{1,1},\dots,x_{1,n},\dots,x_{m,1},\dots,x_{m,n})=p(x_{1,\sigma(1)},\dots,x_{1,\sigma(n)},\dots,x_{m,\sigma(1)},\dots,x_{m,\sigma(n)}),
\end{equation}
and the general linear group $\GL_m$ acts by linear substitutions among the rows,
\begin{equation}\label{eq:GLmactionpolynomials}
g\cdot x_{r,i}:=\sum_{s=1}^m g_{sr}\,x_{s,i},
\qquad g=(g_{sr})\in\GL_m,
\end{equation}
extended to an action by algebra automorphisms. Both are left actions, and they commute, so that $\mathcal{R}_{m,n}$ is a $\GL_m\times S_n$-module. For a polynomial ring $\mathcal{R}$ we denote by $\Mon(\mathcal{R})$ the set of its monomials.
\end{definition}

When $m$ is small, we often drop the first index and use different letters for the rows of variables. For $m=2$, for instance, we write $\mathcal{R}_{2,n}=\C[\mathbf{x},\mathbf{y}]$ with $\mathbf{x}=(x_1,\dots,x_n)$ and $\mathbf{y}=(y_1,\dots,y_n)$, so that $\sigma\cdot p(x_1,\dots,x_n,y_1,\dots,y_n)=p(x_{\sigma(1)},\dots,x_{\sigma(n)},y_{\sigma(1)},\dots,y_{\sigma(n)})$.

\subsubsection*{Representations of $\GL_m$}
We refer to \cite[\S8.2]{Fulton_1996} and \cite{GoodmanWallach2009} for the representation theory of $\GL_m$ used in this paper. Restricting~\eqref{eq:GLmactionpolynomials} to the diagonal torus $(\C^\ast)^m\subseteq\GL_m$ gives $\diag(t_1,\dots,t_m)\cdot x_{r,i}=t_r x_{r,i}$, so the eigenvectors of the torus in $\mathcal{R}_{m,n}$ are the multihomogeneous polynomials, and the weight of a multihomogeneous polynomial $g$ is its multidegree
\begin{equation*}
\mdeg(g)=(\deg_{\mathbf{x}_1}(g),\dots,\deg_{\mathbf{x}_m}(g))\in\N^m.
\end{equation*}
Let $B_m\subseteq\GL_m$ be the Borel subgroup of upper triangular matrices.

\begin{definition}\label{def:hwv}
A nonzero multihomogeneous polynomial $g\in\mathcal{R}_{m,n}$ is a \emph{highest weight vector}, of weight $\mdeg(g)$, if $B_m\cdot g=\C^\ast g$.
\end{definition}

\begin{prop}[{\cite[\S8.2]{Fulton_1996}}]\label{prop:GLmfacts}
Let $M\subseteq\mathcal{R}_{m,n}$ be a finite-dimensional $\GL_m$-submodule.
\begin{enumerate}
    \item If $M$ is irreducible, it contains a unique highest weight vector up to scalar, and its weight is a partition $\nu$ with $\len(\nu)\leq m$. We write $M\cong\mathbb{S}_\nu(\C^m)$.
    \item Two irreducible $\GL_m$-submodules of $\mathcal{R}_{m,n}$ are isomorphic if and only if their highest weight vectors have the same weight.
    \item $M$ is a direct sum of irreducible submodules, and it is generated, as a $\GL_m$-module, by its highest weight vectors.
\end{enumerate}
\end{prop}

The diagonal $S_n$-action~\eqref{eq:actiononpolynomialsgen} is the restriction to permutation matrices of the action of $\GL_n$ on $\mathcal{R}_{m,n}$ by right multiplication on the matrix of variables. The decomposition of $\mathcal{R}_{m,n}$ into irreducible $\GL_m\times\GL_n$-modules is given by the Cauchy decomposition.

\begin{thm}[{Cauchy decomposition, \cite[Corollary~5.6.8]{GoodmanWallach2009}}]\label{thm:cauchy}
For every $d\geq0$, as a $\GL_m\times\GL_n$-module,
\begin{equation*}
(\mathcal{R}_{m,n})_d =\Sym^d(\C^m \otimes \C^n)\cong \bigoplus_{\substack{\nu \vdash d\\ \len(\nu)\leq\min(m,n)}}\mathbb{S}_\nu(\C^m) \otimes \mathbb{S}_\nu(\C^n),
\end{equation*}
where $\mathbb{S}_\nu(\C^n)$ denotes the irreducible $\GL_n$-module with highest weight $\nu$.
\end{thm}

\subsubsection*{Isotypic components}
It is classically known that in characteristic zero every representation of a finite group is completely reducible, i.e., decomposes into a direct sum of irreducible representations. Grouping all isomorphic irreducible representations one obtains the isotypic decomposition (see \cite{serre1977representations} for details). Since $S_n$ acts linearly, every graded piece $(\mathcal{R}_{m,n})_d$ is an $S_n$-representation, and hence we get an \emph{isotypic decomposition} of $\mathcal{R}_{m,n}$,
\begin{equation*}
\mathcal{R}_{m,n}=\bigoplus_{\lambda\vdash n}\mathcal{R}_{m,n,\lambda}.
\end{equation*}
where $\mathcal{R}_{m,n,\lambda}$, the $\lambda$-\emph{isotypic component}, is the sum of all irreducible $S_n$-submodules of $\mathcal{R}_{m,n}$ isomorphic to $S^\lambda$. In particular, $\mathcal{R}_{m,n,\lambda}\cap\mathcal{R}_{m,n,\mu}=\{0\}$ for $\lambda\neq\mu$, and each $\mathcal{R}_{m,n,\lambda}$ is a graded $S_n$-submodule of $\mathcal{R}_{m,n}$.

The commuting $\GL_m$-action leads to a finer decomposition. Restricting the $\GL_n$-action in Theorem~\ref{thm:cauchy} to the subgroup of permutation matrices $S_n\subseteq\GL_n$, each $\mathbb{S}_\nu(\C^n)$ decomposes into irreducible $S_n$-modules. Therefore $(\mathcal{R}_{m,n})_d$ decomposes, as a $\GL_m\times S_n$-module, into irreducible submodules isomorphic to $\mathbb{S}_\nu(\C^m)\otimes S^\lambda$ with $\nu\vdash d$, $\len(\nu)\leq\min(m,n)$ and $\lambda\vdash n$, possibly with multiplicities, and the same holds for every $\GL_m\times S_n$-submodule of $(\mathcal{R}_{m,n})_d$.
\begin{remark}\label{rem:isotypicGL}
For $\lambda\vdash n$, the $\lambda$-isotypic component $\mathcal{R}_{m,n,\lambda}$ is the sum of all irreducible $\GL_m\times S_n$-submodules of $\mathcal{R}_{m,n}$ whose $S_n$-factor is isomorphic to $S^\lambda$; in particular, it is $\GL_m\times S_n$-stable (see also Lemma~\ref{lem:GLminv}).
\end{remark}

\section{Specht ideals}\label{sec:spechtideals}
The Specht ideals in $\mathcal{R}_{1,n}$ were defined in~\cite{moustrou2021symmetric} via Specht polynomials. Let $T\in\Tab(\lambda)$, $\lambda\vdash n$, the \emph{Specht polynomial} associated to $T$ is
\begin{equation*}
\prod_{C\in\Col(T)}\prod_{\substack{i,j\in C\\i\text{ above }j}}(x_i-x_j)\in\mathcal{R}_{1,n},
\end{equation*}
where $\Col(T)$ is the set of columns of $T$.
Equivalently~\cite{Gillespie2024}, up to a nonzero scalar, we define
\begin{equation*}
\spe_T:=\varepsilon_T(\mathbf{x}_T),\qquad\mathbf{x}_T:=\prod_{i=1}^n x_i^{\,r_T(i)-1},
\end{equation*}
where $r_T(i)$ denotes the row of entry $i$ in $T$.

In~\cite{ATY}, the authors define the \emph{higher Specht polynomials}  associated to the pair $(T,S)$ as
\begin{equation*}
\spe_T^S:=\varepsilon_T(\mathbf{x}_T^S),\qquad \mathbf{x}_T^S:=\prod_{i=1}^n x_i^{\,\ccTab(S)(b_T(i))},
\end{equation*}
where $b_T(i)$ denotes the box of $\lambda$ containing entry $i$ in $T$. If $A$ is the superstandard tableau of shape $\lambda$, $\mathbf{x}_T^A=\mathbf{x}_T$ and $\spe_T^A=\spe_T$, recovering the ordinary Specht polynomials. The set $\{\spe_T^S\mid T,S\in\SYT(\lambda),\,\lambda\vdash n\}$ forms a basis of the coinvariant ring $\mathcal{R}_{1,n}/(e_1,\dots,e_n)$, $e_1,\dots,e_n$ being the elementary symmetric polynomials, that is compatible with its decomposition into irreducible $S_n$-modules~\cite{ATY}.

Both $\spe_T$ and $\spe_T^S$ are of the form $\varepsilon_T(p)$. We will see that for any $m\geq1$, any $p\in\mathcal{R}_{m,n}$, and any $T\in\Tab(\lambda)$, $\varepsilon_T(p)\in\mathcal{R}_{m,n,\lambda}$.
The following observation will be used throughout.
\begin{prop}\label{prop:etinvl}
For any $T\in\Tab(\lambda)$, $p\in \mathcal{R}_{m,n}$, $\varepsilon_T(p)\in\mathcal{R}_{m,n,\lambda}$.
\end{prop}
\begin{proof}
Since $\C[S_n]\varepsilon_T\cong S^\lambda$ (see~\cite[\S 2.3]{S01}), the cyclic module generated by $\varepsilon_T(p)$ is $0$ or isomorphic to $S^\lambda$, hence lies in $\mathcal{R}_{m,n,\lambda}$.
\end{proof}

\begin{definition}
The \emph{Specht ideal} of shape $\lambda$ in $\mathcal{R}_{1,n}$ is
\begin{equation*}
I_\lambda:=\langle\spe_T\mid T\in\Tab(\lambda)\rangle
\subseteq\mathcal{R}_{1,n}.
\end{equation*}
\end{definition}
Equivalently~\cite[Proposition~3.6]{Gillespie2024}, $I_\lambda=\langle\spe_T\mid T\in\SYT(\lambda)\rangle$.

\begin{prop}\label{prop:defequiv}
For $m=1$ and any $\lambda\vdash n$,
\begin{equation*}
I_\lambda=\langle\mathcal{R}_{1,n,\lambda}\rangle=\langle\spe_T^S\,\mid\,T,S\in\SYT(\lambda)\rangle.
\end{equation*}
\end{prop}
\begin{proof}
We first show $\langle\spe_T^S\,\mid\,T,S\in\SYT(\lambda)\rangle=\langle\mathcal{R}_{1,n,\lambda}\rangle$. From~\cite{ATY}, we know that
\begin{equation*}
\mathcal{R}_{1,n,\lambda}=\bigoplus_{T,S\in\SYT(\lambda)}\C[e_1,\dots,e_n]\cdot\spe_T^S, 
\end{equation*} where $e_1,\dots,e_n$ are the elementary symmetric polynomials. Thus 
$\langle\mathcal{R}_{1,n,\lambda}\rangle\subseteq \langle\spe_T^S\mid T,S\in\SYT(\lambda)\rangle$.
Conversely, since $\spe_T^S\in\mathcal{R}_{1,n,\lambda}$ by Proposition~\ref{prop:etinvl}, $\langle\spe_T^S\mid T,S\in\SYT(\lambda)\rangle\subseteq\langle\mathcal{R}_{1,n,\lambda}\rangle$.
Secondly, we will prove $I_\lambda=\langle\spe_T^S\,\mid\,T,S\in\SYT(\lambda)\rangle$. Let $A$ be the superstandard tableau of shape $\lambda$. Its cocharge tableau $\ccTab(A)$ is filled with $r-1$ in row $r$, so $\mathbf{x}_T^{A}=\prod_{i=1}^n x_i^{r_T(i)-1}=\mathbf{x}_T$, therefore $\spe_T^{A}=\varepsilon_T(\mathbf{x}_T)=\spe_T$, and $I_\lambda\subseteq\langle\spe_T^S\mid T,S\in\SYT(\lambda)\rangle$. Conversely, for any higher Specht polynomial $\spe_T^S$ and any transposition $\rho=(ij)\in C(T)$, we have 
\begin{equation*}
\rho\cdot \spe_T^S=\rho\cdot\left(\sum_{\tau\in C(T)}\sgn(\tau)\tau\right)\cdot \left(\sum_{\sigma\in R(T)}\sigma \right)(\mathbf{x}_T^S)=-\left(\sum_{\rho\tau\in C(T)}\sgn(\rho\tau)\rho\tau\right)\cdot \left(\sum_{\sigma\in R(T)}\sigma \right)(\mathbf{x}_T^S)=-\spe_T^S,
\end{equation*}
so $(x_i-x_j)\mid\spe_T^S$ for every pair $i,j$ in the same column of $T$, and hence the product of the corresponding distinct linear factors --- namely $\spe_T$, up to a nonzero scalar --- divides $\spe_T^S$. Therefore $\spe_T^S\in I_\lambda$.
\end{proof}

For $m=1$, Proposition~\ref{prop:defequiv} shows that the usual Specht ideal can be recovered from the $\lambda$-isotypic component of the polynomial ring.  When $m\geq2$, there is no single distinguished family of ordinary Specht polynomials playing the same role.  The isotypic description, however, remains available and is compatible with the commuting $\GL_m$-action.  This leads to the following definition.
\begin{definition}\label{def:diagspecht}
We define the \emph{diagonal Specht ideal} of shape $\lambda$ as
\begin{equation*}
I_\lambda:=\left\langle \mathcal{R}_{m,n,\lambda}\right\rangle.
\end{equation*}
\end{definition}
\begin{remark}\label{rem:char0}
    All ideals considered in this paper are defined over $\Q$. If $\K/\Q$ is a field extension, then extension of scalars commutes with taking radicals: \[\rad(I)\,\overline{\K}[\mathbf{x}_1,\dots,\mathbf{x}_m] = \rad\bigl(I\, \overline{\K}[\mathbf{x}_1,\dots,\mathbf{x}_m]\bigr).\] Thus radicality and non-radicality of the ideals considered below are unchanged after extending the ground field within characteristic zero.
\end{remark}
The projection onto $\mathcal{R}_{m,n,\lambda}$ is given by the following standard result.
\begin{thm}
[{\cite[Theorem~8]{serre1977representations}}]\label{thm:serre}
Let $G$ be a finite group, $\rho: G\longrightarrow \GL(M)$ be a linear representation of $G$, and let $M=M_1\oplus\cdots\oplus M_k$ be its isotypic decomposition, where $\chi_i$ and $n_i$ are the character and dimension, respectively, of the irreducible constituent of $M_i$. Then the projection $\pi_i$ into $M_i$ is given by
\begin{equation*}
\pi_i=\frac{n_i}{|G|}\sum_{g\in G}\chi_i(g^{-1})\rho(g).
\end{equation*}
\end{thm}
Consider the module $\mathcal{R}_{m,n}$ with the $S_n$-action \eqref{eq:actiononpolynomialsgen}, let $\chi_\lambda$ be the character of the irreducible representation corresponding to $\lambda$, let $n_\lambda$ be its dimension, and $\pi_\lambda$ be the projection onto $\mathcal{R}_{m,n,\lambda}$. Note that $\mathcal{R}_{m,n}=\bigoplus_{d\in\N}(\mathcal{R}_{m,n})_d$ is a direct sum of finite-dimensional $S_n$-modules, therefore we can apply the projection of the theorem degree by degree, and $\pi_\lambda$ is the direct sum of the projections onto each degree $d$ component $(\mathcal{R}_{m,n,\lambda})_d$ of $\mathcal{R}_{m,n,\lambda}$.
\begin{lemma}\label{lem:GLminv}
The isotypic components $\mathcal{R}_{m,n,\lambda}$ are $\GL_m\times S_n$-stable.
\end{lemma}
\begin{proof}
$\mathcal{R}_{m,n,\lambda}$ are $S_n$-stable by definition. We want to prove that if $g\in\GL_m$ and $f\in\mathcal{R}_{m,n,\lambda}$, then $g\cdot f\in\mathcal{R}_{m,n,\lambda}$. Since the two actions commute, we have $\pi_\lambda(g\cdot f)=g\cdot\pi_\lambda(f)$, and since $\pi_\lambda(f)=f$, $\pi_\lambda(g\cdot f)=g\cdot f$, therefore $g\cdot f\in\mathcal{R}_{m,n,\lambda}$ and $\mathcal{R}_{m,n,\lambda}$ is $\GL_m$-stable.
\end{proof}

The following result relates the projection $\pi_\lambda$ to the Young symmetrizers.
\begin{thm}[{\cite[Theorem 1.6]{GarsiaE2020}}]\label{thm:theorem16}
Let $\lambda\vdash n$, then
\begin{equation*}
\sum_{\sigma\in  S_n}\chi_\lambda(\sigma)\sigma=\sum_{T\in\Tab(\lambda)}\frac{n_\lambda}{n!}\varepsilon_T.
\end{equation*}
\end{thm}

As a consequence, $I_\lambda$ is generated by the polynomials $\varepsilon_T(p)$ with $p$ a monomial.
\begin{prop}\label{prop:genIlambda}
Let $\lambda\vdash n$. Then
\begin{equation*}
I_\lambda=\langle\varepsilon_T(p)\mid p\in\Mon(\mathcal{R}_{m,n}),\,T\in\Tab(\lambda)\rangle.
\end{equation*}
\end{prop}
\begin{proof}
The inclusion $\supseteq$ follows from Proposition~\ref{prop:etinvl}. For $f\in I_\lambda$, by definition $f=\sum_i g_if_i$ where $f_i\in\mathcal{R}_{m,n,\lambda}$ and $g_i\in\mathcal{R}_{m,n}$, and since $f_i=\pi_\lambda(f_i)$, by Theorem~\ref{thm:theorem16} each $f_i$ is a linear combination of $\varepsilon_T(f_i)$, which in turn is a linear combination of $\varepsilon_T(p)$ where $p$ are the monomials of the expansion of $f_i$. Therefore $I_\lambda\subseteq\langle\varepsilon_T(p)\,\mid\,p\in\Mon(\mathcal{R}_{m,n}),\,T\in\Tab(\lambda)\rangle$.
\end{proof}

\section{Specht varieties}\label{sec:varieties}
In this section we extend the characterization of Specht varieties from~\cite{moustrou2021symmetric} to the diagonal setting (Theorem~\ref{thm:charzeros}). As consequences we describe the irreducible components of $V_\lambda$ and the radical of $I_\lambda$ (Corollary~\ref{cor:linearspaces}), and we show that $\lambda\mapsto V_\lambda$ is an isomorphism of posets from $(\partn,\trianglelefteq)$ to $(\{V_\lambda:\lambda\in\partn\},\supseteq)$, exactly as for $m=1$ (Corollary~\ref{cor:posetvar}). In the second part of the section we compute the dimension of $V_\lambda$ and characterize the partitions $\lambda$ for which $V_\lambda$ is equidimensional.

\subsection{The zero set and its irreducible components}

For an ideal $J\subseteq\mathcal{R}_{m,n}$ we write $V(J)\subseteq\C^{mn}$ for its zero set, and for a subset $X\subseteq\C^{mn}$ we write $I(X)\subseteq\mathcal{R}_{m,n}$ for its vanishing ideal.

\begin{definition}\label{def:spechtvariety}
The \emph{Specht variety} of shape $\lambda\vdash n$ is the zero set
\begin{equation*}
V_\lambda:=V(I_\lambda)\subseteq\C^{mn}
\end{equation*}
of the diagonal Specht ideal $I_\lambda\subseteq\mathcal{R}_{m,n}$. When the number of sets of variables needs to be emphasized, we write $I_\lambda^{(m)}$ and $V_\lambda^{(m)}$.
\end{definition}

\begin{definition}\label{def:Hmu}
We identify a point $P\in\C^{mn}$ with the $m\times n$ matrix $(u_{r,i})$ and write $P=(\mathbf{u}_{\bullet,1},\dots,\mathbf{u}_{\bullet,n})$, where $\mathbf{u}_{\bullet,i}=(u_{1,i},\dots,u_{m,i})^T\in\C^m$ is the $i$-th column of $P$. For $\mu\vdash n$, let $H_\mu\subseteq\C^{mn}$ (or $H_\mu^{(m)}$) be the set of points whose columns take exactly $\len(\mu)$ distinct values, with multiplicities $\mu_1,\dots,\mu_{\len(\mu)}$. The sets $H_\mu$, $\mu\vdash n$, are nonempty and partition $\C^{mn}$; we say that $P$ has \emph{type} $\mu$ if $P\in H_\mu$.
\end{definition}

\begin{definition}\label{def:Wpippi}
For a set partition $\Pi$ of $[n]$, let
\begin{equation*}
W_\Pi:=\{P\in\C^{mn}\mid\mathbf{u}_{\bullet,i}=\mathbf{u}_{\bullet,j}\text{ whenever }i\sim_\Pi j\},\quad
\mathfrak{p}_\Pi:=\langle x_{r,i}-x_{r,j}\mid 1\leq r\leq m,\ i\sim_\Pi j\rangle\subseteq\mathcal{R}_{m,n}.
\end{equation*}
\end{definition}

\begin{remark}[Actions on points]\label{rmk:piinstab}
The group $S_n$ acts on $\C^{mn}$ by permuting the columns, and $\GL_m$ by left multiplication:
\begin{equation}\label{eq:actiononpoints}
\sigma\cdot(\mathbf{u}_{\bullet,1},\dots,\mathbf{u}_{\bullet,n}):=(\mathbf{u}_{\bullet,\sigma^{-1}(1)},\dots,\mathbf{u}_{\bullet,\sigma^{-1}(n)}),\qquad\sigma\in S_n,
\end{equation}
\begin{equation}\label{actionGLmonpoints}
g\cdot(\mathbf{u}_{\bullet,1},\dots,\mathbf{u}_{\bullet,n}):=(g\mathbf{u}_{\bullet,1},\dots,g\mathbf{u}_{\bullet,n}),\qquad g\in\GL_m.
\end{equation}
These are compatible with the actions~\eqref{eq:actiononpolynomialsgen} and~\eqref{eq:GLmactionpolynomials} on $\mathcal{R}_{m,n}$ in the following sense: for all $q\in\mathcal{R}_{m,n}$, $P\in\C^{mn}$, $\sigma\in S_n$ and $g\in\GL_m$,
\begin{equation}\label{eq:evalactions}
(\sigma\cdot q)(P)=q(\sigma^{-1}\cdot P)
\qquad\text{and}\qquad
(g\cdot q)(P)=q(g^{T}\cdot P).
\end{equation}
In particular, if $\pi\in\Stab_{S_n}(P)$, then $(\pi\cdot q)(P)=q(P)$.
\end{remark}

We collect the elementary properties of the strata $H_\mu$ and of the subspaces $W_\Pi$ that will be used repeatedly.

\begin{lemma}\label{lem:strata}
Let $\mu,\nu\vdash n$, and let $\Pi,\Theta$ be set partitions of $[n]$.
\begin{enumerate}
    \item Each $H_\mu$ is $\GL_m\times S_n$-stable.
    \item $W_\Pi=V(\mathfrak{p}_\Pi)$ is a linear subspace of dimension $m\len(\sh(\Pi))$, the ideal $\mathfrak{p}_\Pi$ is prime, and $I(W_\Pi)=\mathfrak{p}_\Pi$. Moreover, $W_\Theta\subseteq W_\Pi$ if and only if $\Pi\prec\Theta$.
    \item $\overline{H_\mu}=\bigcup_{\Pi\in\SP_\mu}W_\Pi$, and $H_\nu\subseteq\overline{H_\mu}$ if and only if $\mu\prec\nu$.
\end{enumerate}
\end{lemma}
\begin{proof}
(1) For $g\in\GL_m$ we have $g\mathbf{u}_{\bullet,i}=g\mathbf{u}_{\bullet,j}$ if and only if $\mathbf{u}_{\bullet,i}=\mathbf{u}_{\bullet,j}$, and $\sigma\in S_n$ merely permutes the columns; so both actions preserve the type of a point.

(2) The subspace $W_\Pi$ is cut out by the linear forms generating $\mathfrak{p}_\Pi$, so $W_\Pi=V(\mathfrak{p}_\Pi)$, and a point of $W_\Pi$ is determined by one free vector in $\C^m$ for each block of $\Pi$, which gives the dimension. An ideal generated by linear forms is prime, so $I(W_\Pi)=\mathfrak{p}_\Pi$ by the Nullstellensatz. Finally, $W_\Theta\subseteq W_\Pi$ means that every identification $\mathbf{u}_{\bullet,i}=\mathbf{u}_{\bullet,j}$ imposed by $\Pi$ is also imposed by $\Theta$, that is, $i\sim_\Pi j$ implies $i\sim_\Theta j$.

(3) A point $P$ lies in $H_\mu$ if and only if $P\in W_\Pi$ for some $\Pi\in\SP_\mu$ and the columns of $P$ indexed by different blocks of $\Pi$ are pairwise distinct. Hence $H_\mu\subseteq\bigcup_{\Pi\in\SP_\mu}W_\Pi$, and since the right-hand side is closed, so is $\overline{H_\mu}$. Conversely, for $\Pi\in\SP_\mu$ the set $H_\mu\cap W_\Pi$ is the complement in $W_\Pi$ of finitely many proper linear subspaces, hence dense in $W_\Pi$, so $W_\Pi\subseteq\overline{H_\mu}$. For the last claim, if $\mu\prec\nu$ and $P\in H_\nu$, then $P\in W_\Theta$ for some $\Theta\in\SP_\nu$, and splitting the blocks of $\Theta$ yields $\Pi\in\SP_\mu$ with $\Pi\prec\Theta$; thus $P\in W_\Theta\subseteq W_\Pi\subseteq\overline{H_\mu}$. If conversely $H_\nu\subseteq\overline{H_\mu}$, pick $P\in H_\nu$ and $\Pi\in\SP_\mu$ with $P\in W_\Pi$; the type $\nu$ of $P$ is then obtained from $\sh(\Pi)=\mu$ by merging parts, that is, $\mu\prec\nu$.
\end{proof}

The next two lemmas give the inclusion $\supseteq$ in Theorem~\ref{thm:charzeros}.

\begin{lemma}\label{lem:cstab}
Let $\lambda,\mu\vdash n$ with $\mu\not\trianglelefteq\lambda$, let $P\in H_\mu$, and let $T\in\Tab(\lambda)$. Then $C(T)\cap\Stab_{S_n}(P)$ contains a transposition.
\end{lemma}
\begin{proof}
Let $\mathbf{a}_1,\dots,\mathbf{a}_{\len(\mu)}$ be the distinct values of the columns of $P$, numbered so that $\mathbf{a}_j$ occurs $\mu_j$ times, and let $B_j:=\{i\in[n]\mid\mathbf{u}_{\bullet,i}=\mathbf{a}_j\}$, so that $|B_j|=\mu_j$. A transposition $(i_1\,i_2)$ with $i_1,i_2\in B_j$ fixes $P$, and it lies in $C(T)$ if $i_1$ and $i_2$ occupy the same column of $T$. Suppose, for a contradiction, that no two elements of the same $B_j$ lie in the same column of $T$. Fix $1\leq k\leq\len(\mu)$. The elements of $B_1\cup\dots\cup B_k$ occupy $\sum_{j\leq k}\mu_j$ boxes of $T$, and a column of height $h$ contains at most $\min(k,h)$ of them, at most one from each $B_j$. Summing over the columns of $T$ gives
\begin{equation*}
\sum_{j\leq k}\mu_j\leq\sum_{c}\min(k,h_c)=\sum_{i\leq k}\lambda_i .
\end{equation*}
As this holds for every $k$, we obtain $\mu\trianglelefteq\lambda$, a contradiction.
\end{proof}

\begin{lemma}\label{lem:projectionyoung}
Let $\lambda,\mu\vdash n$ with $\mu\not\trianglelefteq\lambda$, and let $P\in H_\mu$. Then for all $f\in\mathcal{R}_{m,n}$ and all $T\in\Tab(\lambda)$,
\begin{equation*}
(\alpha(T)\cdot f)(P)=0
\qquad\text{and}\qquad
(\varepsilon_T\cdot f)(P)=0 .
\end{equation*}
\end{lemma}
\begin{proof}
By Lemma~\ref{lem:cstab} there is a transposition $\pi\in C(T)$ fixing $P$. Then $\tau\mapsto\pi\tau$ is a bijection from the even to the odd elements of $C(T)$, and $(\pi\tau\cdot f)(P)=(\tau\cdot f)(P)$ for every $\tau$ by Remark~\ref{rmk:piinstab}. Hence the contributions of the even and of the odd permutations to $(\alpha(T)\cdot f)(P)$ cancel, which proves the first identity. The second follows by applying the first to $\beta(T)\cdot f$, since $\varepsilon_T=\alpha(T)\beta(T)$.
\end{proof}

For $m=1$, the zero set of the Specht ideal was determined in~\cite[Corollary~1]{moustrou2021symmetric}. The following theorem extends this description to every $m\geq1$.

\begin{thm}\label{thm:charzeros}
For every $m\geq1$ and $\lambda\vdash n$,
\begin{equation*}
V_\lambda=\bigcup_{\mu\not\trianglelefteq\lambda}H_\mu .
\end{equation*}
\end{thm}
\begin{proof}
For the inclusion $\supseteq$, let $P\in H_\mu$ with $\mu\not\trianglelefteq\lambda$. By Proposition~\ref{prop:genIlambda}, $I_\lambda$ is generated by the polynomials $\varepsilon_T\cdot p$ with $T\in\Tab(\lambda)$ and $p$ a monomial, and these vanish at $P$ by Lemma~\ref{lem:projectionyoung}. Hence $P\in V_\lambda$.

For the inclusion $\subseteq$, let $P\in H_\mu$ with $\mu\trianglelefteq\lambda$; we show that $P\notin V_\lambda$. Write $\mathcal{R}_{1,n}=\C[\mathbf{z}]$ with $\mathbf{z}=(z_1,\dots,z_n)$, and let $\mathbf{a}_1,\dots,\mathbf{a}_{\len(\mu)}$ be the distinct values of the columns of $P$. Choose $t=(t_1,\dots,t_m)\in\C^m$ outside the finitely many hyperplanes $\langle t,\mathbf{a}_j-\mathbf{a}_{j'}\rangle=0$, $j\neq j'$, and consider the algebra homomorphism
\begin{equation*}
\psi:\mathcal{R}_{1,n}\longrightarrow\mathcal{R}_{m,n},\qquad z_i\longmapsto\sum_{r=1}^m t_r x_{r,i}.
\end{equation*}
Since $\psi(\sigma\cdot z_i)=\psi(z_{\sigma(i)})=\sigma\cdot\psi(z_i)$ for all $\sigma\in S_n$ and all $i$, the map $\psi$ is $S_n$-equivariant; in particular $\psi(\varepsilon_T\cdot f)=\varepsilon_T\cdot\psi(f)$ for all $f\in\mathcal{R}_{1,n}$ and $T\in\Tab(\lambda)$. Moreover, evaluating at $P$ gives
\begin{equation*}
\psi(f)(P)=f(Q)\quad\text{for all }f\in\mathcal{R}_{1,n},
\qquad\text{where }Q:=(w_1,\dots,w_n)\in\C^n,\ \ w_i:=\sum_{r=1}^m t_r u_{r,i}.
\end{equation*}
By the choice of $t$, the values $\langle t,\mathbf{a}_1\rangle,\dots,\langle t,\mathbf{a}_{\len(\mu)}\rangle$ are pairwise distinct, so $Q\in H_\mu^{(1)}$, and therefore $Q\notin V_\lambda^{(1)}$ by~\cite[Corollary~1]{moustrou2021symmetric}. By Proposition~\ref{prop:genIlambda} there exist $T\in\Tab(\lambda)$ and a monomial $g\in\mathcal{R}_{1,n}$ with $(\varepsilon_T\cdot g)(Q)\neq0$. Then
\begin{equation*}
(\varepsilon_T\cdot\psi(g))(P)=\psi(\varepsilon_T\cdot g)(P)=(\varepsilon_T\cdot g)(Q)\neq0,
\end{equation*}
and $\varepsilon_T\cdot\psi(g)\in I_\lambda$ by Proposition~\ref{prop:etinvl}. Hence $P\notin V_\lambda$.
\end{proof}

\begin{definition}\label{def:Alambda}
For $\lambda\vdash n$, let $\mathcal{A}_\lambda$ be the set of minimal elements, with respect to the refinement order $\prec$, of the set $\{\mu\vdash n:\mu\not\trianglelefteq\lambda\}$. Since this set is finite, every $\mu\not\trianglelefteq\lambda$ is refined by some element of $\mathcal{A}_\lambda$.
\end{definition}

Since $V_\lambda$ is closed, Theorem~\ref{thm:charzeros} can be rewritten in terms of the closures of the minimal strata.

\begin{corollary}\label{cor:closure}
For every $\lambda\vdash n$,
\begin{equation*}
V_\lambda=\bigcup_{\mu\in\mathcal{A}_\lambda}\overline{H_\mu}.
\end{equation*}
\end{corollary}
\begin{proof}
As $V_\lambda$ is closed, Theorem~\ref{thm:charzeros} gives $V_\lambda=\bigcup_{\mu\not\trianglelefteq\lambda}\overline{H_\mu}$. Every $\mu\not\trianglelefteq\lambda$ is refined by some $\nu\in\mathcal{A}_\lambda$, and then $\overline{H_\mu}\subseteq\overline{H_\nu}$ by Lemma~\ref{lem:strata}(3), so the union may be restricted to $\mathcal{A}_\lambda$.
\end{proof}

\begin{example}\label{ex:A321}
Let $n=6$ and $\lambda=(3,2,1)$. From the Hasse diagram in Example~\ref{ex:dominance6}, the partitions not dominated by $(3,2,1)$ are $(4,1,1)$, $(3,3)$, $(4,2)$, $(5,1)$ and $(6)$, and among them $(4,1,1)$ and $(3,3)$ are the minimal ones with respect to refinement. Hence $\mathcal{A}_{(3,2,1)}=\{(4,1,1),(3,3)\}$ and
\begin{equation*}
V_{(3,2,1)}=\overline{H_{(4,1,1)}}\cup\overline{H_{(3,3)}}
\end{equation*}
for every $m$. For $m=1$, for instance, the point $P=(a,a,a,a,b,b)$ with $a\neq b$ has type $(4,2)$ and lies in $\overline{H_{(4,1,1)}}$.
\end{example}

\begin{corollary}\label{cor:GLminvariantvariety}
$V_\lambda$ and $I(V_\lambda)$ are $\GL_m\times S_n$-stable. In particular, $I(V_\lambda)$ is multigraded.
\end{corollary}
\begin{proof}
By Theorem~\ref{thm:charzeros} and Lemma~\ref{lem:strata}(1), $V_\lambda$ is a union of $\GL_m\times S_n$-stable sets. For $f\in I(V_\lambda)$ and $P\in V_\lambda$, \eqref{eq:evalactions} gives $(g\cdot f)(P)=f(g^{T}\cdot P)=0$ and $(\sigma\cdot f)(P)=f(\sigma^{-1}\cdot P)=0$, so $I(V_\lambda)$ is stable as well. Stability under the diagonal torus of $\GL_m$ means that $I(V_\lambda)$ is multigraded.
\end{proof}

\begin{corollary}\label{cor:linearspaces}
Let $\lambda\vdash n$. Then
\begin{equation*}
V_\lambda=\bigcup_{\mu\in\mathcal{A}_\lambda}\bigcup_{\Pi\in\SP_\mu}W_\Pi ,
\end{equation*}
and the subspaces $W_\Pi$ with $\sh(\Pi)\in\mathcal{A}_\lambda$ are precisely the irreducible components of $V_\lambda$; in particular, $V_\lambda$ is a finite union of linear subspaces, and $\dim W_\Pi=m\len(\sh(\Pi))$. Moreover,
\begin{equation*}
\rad(I_\lambda)=\bigcap_{\mu\in\mathcal{A}_\lambda}\bigcap_{\Pi\in\SP_\mu}\mathfrak{p}_\Pi ,
\end{equation*}
and the $\mathfrak{p}_\Pi$ with $\sh(\Pi)\in\mathcal{A}_\lambda$ are the minimal primes of $I_\lambda$ and of $\rad(I_\lambda)$.
\end{corollary}
\begin{proof}
The first equality follows from Corollary~\ref{cor:closure} and Lemma~\ref{lem:strata}(3). Each $W_\Pi$ is irreducible, so the irreducible components of $V_\lambda$ are the maximal ones among the $W_\Pi$ with $\sh(\Pi)\in\mathcal{A}_\lambda$. If $W_\Theta\subsetneq W_\Pi$ for two such subspaces, then $\Pi\prec\Theta$ with $\Pi\neq\Theta$ by Lemma~\ref{lem:strata}(2), hence $\sh(\Pi)\prec\sh(\Theta)$ with $\sh(\Pi)\neq\sh(\Theta)$, contradicting the minimality of $\sh(\Theta)$. So there are no containments among them, all of them are components, and their dimensions are given by Lemma~\ref{lem:strata}(2).

By the Nullstellensatz, $\rad(I_\lambda)=I(V_\lambda)$, which is the intersection of the vanishing ideals $I(W_\Pi)=\mathfrak{p}_\Pi$ of the components. These are the minimal primes of $I(V_\lambda)=\rad(I_\lambda)$, and $I_\lambda$ has the same minimal primes as its radical.
\end{proof}

\begin{corollary}\label{cor:posetvar}
The map $\lambda\mapsto V_\lambda$ is an isomorphism of posets from $(\partn,\trianglelefteq)$ to $(\{V_\lambda:\lambda\in\partn\},\supseteq)$.
\end{corollary}
\begin{proof}
If $\mu\trianglelefteq\lambda$, then $\{\nu:\nu\not\trianglelefteq\lambda\}\subseteq\{\nu:\nu\not\trianglelefteq\mu\}$, so $V_\lambda\subseteq V_\mu$ by Theorem~\ref{thm:charzeros}. Conversely, suppose that $V_\lambda\subseteq V_\mu$ and pick $P\in H_\mu$. By Theorem~\ref{thm:charzeros}, $P\notin V_\mu$, hence $P\notin V_\lambda$, hence $\mu\trianglelefteq\lambda$. Thus $\mu\trianglelefteq\lambda$ if and only if $V_\lambda\subseteq V_\mu$; in particular the map is injective, and it is surjective by definition.
\end{proof}

We will see in Corollary~\ref{cor:posetideals} that, in contrast, $I_\mu\not\subseteq I_\lambda$ may happen for $\mu\trianglelefteq\lambda$ when $m\geq2$.

\subsection{Dimension of the variety}
Having identified the irreducible components as linear subspaces indexed by minimal types, we can read off the dimension of $V_\lambda$ from the lengths of the corresponding partitions.  We next determine the largest possible such length and characterize equidimensionality.
\begin{lemma}\label{lem:length}
Let $\lambda=(\lambda_1,\ldots,\lambda_\ell)$ and $\mu=(\mu_1,\ldots,\mu_s)$ be partitions of $n$. If $\mu\not\trianglelefteq\lambda$, then $s\leq n-\lambda_1$. Moreover, if $\lambda_2>1$ and $s=n-\lambda_1$, then $\mu=\lambda^\sharp:=(\lambda_1+1,1^{n-\lambda_1-1})$.
\end{lemma}
\begin{proof}
Since $\mu\not\trianglelefteq\lambda$, there exists an integer $k\geq1$ such that $\sum_{i=1}^k\mu_i>\sum_{i=1}^k\lambda_i$, and necessarily $k<\ell$. Since $\lambda_2,\ldots,\lambda_k>0$, we have
$\sum_{i=1}^k\lambda_i\geq\lambda_1+k-1$, and 
\begin{equation*}
\sum_{i=1}^k\mu_i\geq\sum_{i=1}^k\lambda_i+1\geq\lambda_1+k.
\end{equation*}
If $s>k$, then $\sum_{i=k+1}^s\mu_i=n-\sum_{i=1}^k\mu_i\leq n-\lambda_1-k$, and since each of the $s-k$ remaining parts is positive,
\begin{equation*}
s-k\leq\sum_{i=k+1}^s\mu_i\leq n-\lambda_1-k,
\end{equation*}
hence $s\leq n-\lambda_1$. If $s\leq k$, since $k<\ell$, the parts $\lambda_2,\ldots,\lambda_{k+1}$ are all positive, so $n=\sum_{i=1}^{\ell}\lambda_i\geq\lambda_1+k$ and $s\leq k\leq n-\lambda_1$. Therefore, in all cases, $s\leq n-\lambda_1$.

Suppose now that $\lambda_2\geq2$, and that $s=n-\lambda_1$. Then $\sum_{i=1}^s(\mu_i-1)=n-s=\lambda_1$, so $\sum_{i=1}^k\mu_i\leq k+\lambda_1$ for every $k$, with equality only if $\mu_i=1$ for all $i>k$. Let $k$ be such that $\sum_{i=1}^k\mu_i>\sum_{i=1}^k\lambda_i$. If $k\geq2$, then $\sum_{i=1}^k\lambda_i\geq\lambda_1+\lambda_2+(k-2)\geq\lambda_1+k\geq\sum_{i=1}^k\mu_i$, a contradiction. Hence $k=1$, so $\lambda_1<\mu_1\leq\lambda_1+1$, which forces $\mu_1=\lambda_1+1$ and $\mu_i=1$ for $i>1$, that is, $\mu=\lambda^\sharp$.
\end{proof}

For hook partitions, $\mu\not\trianglelefteq\lambda$ depends only on $\len(\mu)$, and $\mathcal{A}_\lambda$ can be easily described. We will use this in Section~\ref{sec:radicality}.
\begin{lemma}\label{lem:lenhook}
Let $\lambda$ be a hook partition, and $\mu$ a partition of general shape. Then $\mu\not\trianglelefteq\lambda$ if and only if $\len(\mu)<\len(\lambda)$.
\end{lemma}
\begin{proof}
Assume $\len(\mu)<\len(\lambda)$, then $\sum_{i=1}^{\len(\mu)}\mu_i=n>\sum_{i=1}^{\len(\mu)}\lambda_i$, therefore $\mu\not\trianglelefteq\lambda$. Assume $\mu\not\trianglelefteq\lambda$, then $\sum_{i=1}^{k}\mu_i>\sum_{i=1}^{k}\lambda_i$ for some $k$, $1\leq k\leq \len(\lambda)$, then $n-\sum_{i=1}^{k}\mu_i<n-\sum_{i=1}^{k}\lambda_i$. Lastly, $\len(\mu)\leq k+n-\sum_{i=1}^{k}\mu_i<\len(\lambda)=k+n-\sum_{i=1}^{k}\lambda_i$.
\end{proof}

\begin{lemma}\label{lem:alambdahooks}
Let $\lambda$ be a hook partition, then $\mathcal{A}_\lambda=\{\mu,\,\len(\mu)=\len(\lambda)-1\}$.
\end{lemma}
\begin{proof}
Let $\lambda=(a,1^{n-a})$, $a>0$, so $\len(\lambda)=n-a+1$. By Lemma~\ref{lem:lenhook}, $\{\mu\not\trianglelefteq\lambda\}=\{\mu,\,\len(\mu)<n-a+1\}$. If $\len(\mu)<n-a<n$, then $\mu_i\geq 2$ for some $i$. Splitting $\mu_i$, we get a new partition $\mu'$ with $\mu'\prec\mu$ and $\len(\mu')<n-a+1$, therefore $\mu$ is not minimal. Let $\mu$ be such that $\len(\mu)=n-a$, and $\mu'$ be such that $\mu'\prec\mu$, $\mu'\neq\mu$, then $\len(\mu')>\len(\mu)=n-a$, and $\mu'\trianglelefteq\lambda$, therefore $\mu$ is minimal in $\{\mu\not\trianglelefteq\lambda\}$.
\end{proof}

\begin{lemma}\label{lem:lambdasharp}
Let $\lambda\vdash n$, $\lambda\neq(n)$, and let $\lambda^\sharp=(\lambda_1+1,1^{n-\lambda_1-1})$. Then $\lambda^\sharp\in\mathcal{A}_\lambda$ and $\len(\lambda^\sharp)=n-\lambda_1$.
\end{lemma}
\begin{proof}
We have $\lambda^\sharp\not\trianglelefteq\lambda$, since $\lambda^\sharp_1=\lambda_1+1>\lambda_1$, and clearly
$\len(\lambda^\sharp)=n-\lambda_1$. If $\mu\prec\lambda^\sharp$ with $\mu\neq\lambda^\sharp$, then $\len(\mu)>n-\lambda_1$, so $\mu\trianglelefteq\lambda$ by Lemma~\ref{lem:length}. Hence $\lambda^\sharp$ is minimal in $\{\mu\not\trianglelefteq\lambda\}$, that is, $\lambda^\sharp\in\mathcal{A}_\lambda$.
\end{proof}

\begin{prop}\label{prop:dimension}
Let $\lambda\vdash n$, $\lambda\neq(n)$. Then $\dim V_\lambda=m(n-\lambda_1)$.
\end{prop}
\begin{proof}
By Corollary~\ref{cor:linearspaces}, the irreducible components of $V_\lambda$ are $W_\Pi$ with $\sh(\Pi)\in\mathcal{A}_\lambda$, which have dimension $m\len(\sh(\Pi))$, so $\dim V_\lambda=m\max\{\len(\mu)\,\mid\,\mu\in\mathcal{A}_\lambda\}$, and by Lemmas~\ref{lem:length} and~\ref{lem:lambdasharp}, $m\max\{\len(\mu)\,\mid\,\mu\in\mathcal{A}_\lambda\}=m\len(\lambda^\sharp)=m(n-\lambda_1)$.
\end{proof}
\begin{prop}\label{prop:equidim}
Let $\lambda=(\lambda_1,\dots,\lambda_\ell)\vdash n$, $\lambda\neq(n)$. Then the variety $V_\lambda$ is equidimensional in exactly the following cases:
\begin{enumerate}
    \item $\lambda_2=1$,
    \item $\lambda_{\ell-1}=\lambda_1$.
\end{enumerate}
\end{prop}
\begin{proof}
Proposition~\ref{prop:dimension} shows that $\dim V_\lambda=m(n-\lambda_1)$. By Corollary~\ref{cor:linearspaces}, the irreducible components of $V_\lambda$ are $W_\Pi$ with $\sh(\Pi)\in\mathcal{A}_\lambda$, which have dimension $m\len(\sh(\Pi))$, so $V_\lambda$ is equidimensional if and only if for all $\nu\in\mathcal{A}_\lambda$, $\len(\nu)=n-\lambda_1$.
\begin{enumerate}
    \item Let $\lambda_2=1$. By Lemma~\ref{lem:alambdahooks}, every $\nu\in\mathcal{A}_\lambda$ has $\len(\nu)=\len(\lambda)-1=n-\lambda_1$, so $V_\lambda$ is equidimensional.
    \item Let $\lambda_2>1$. Assume first that $\lambda_{\ell-1}<\lambda_1$, and let $t$ be minimal such that $\lambda_t<\lambda_1$. Then $\nu=(\lambda_1,\cdots,\lambda_{t-1},\lambda_t+1,1,\cdots,1)\not\trianglelefteq\lambda$, since $\sum_{i=1}^t\nu_i=\sum_{i=1}^t\lambda_i+1$. Let $\nu'\in\mathcal{A}_\lambda$ with $\nu'\prec\nu$. Then $\nu'_1\leq\nu_1=\lambda_1$, so $\nu'\neq\lambda^\sharp$, and $\len(\nu')<n-\lambda_1$ by Lemma~\ref{lem:length}, so that $V_\lambda$ is not equidimensional. Assume now that $\lambda_{\ell-1}=\lambda_1$, that is, $\lambda_1=\cdots=\lambda_{\ell-1}\geq\lambda_\ell\geq1$. Let $\nu\not\trianglelefteq\lambda$, and let $t$ be such that
    \begin{equation*}
    \sum_{i=1}^t\lambda_i<\sum_{i=1}^t\nu_i.
    \end{equation*}
    As always, $t<\ell$ necessarily. If $\nu_1\leq\lambda_1$, then
    \begin{equation*}
    t\nu_1\geq\sum_{i=1}^t\nu_i>\sum_{i=1}^t\lambda_i=t\lambda_1,
    \end{equation*}
    a contradiction, hence $\nu_1>\lambda_1$. Then, setting $s_0:=0$ and $s_j:=\nu_1+\dots+\nu_j-\lambda_1$ for $j\geq1$, we can write
    \begin{equation*}
    \nu_j=\sum_{i=s_{j-1}+1}^{s_j}\lambda^\sharp_i\qquad\text{for all }j\geq1,
    \end{equation*}
    so that $\lambda^\sharp\prec\nu$. Since this holds for every $\nu\not\trianglelefteq\lambda$, and $\lambda^\sharp\in\mathcal{A}_\lambda$ by Lemma~\ref{lem:lambdasharp}, we get $\mathcal{A}_\lambda=\{\lambda^\sharp\}$, i.e.\ $V_\lambda$ is equidimensional.
\end{enumerate}
\end{proof}

\section{Radicality of Specht ideals}\label{sec:radicality}

In this section we study for which $\lambda$, $m$ and $n$ the ideal $I_\lambda\subseteq\mathcal{R}_{m,n}$ is radical. For $m=1$, Specht ideals are radical for all $\lambda$~\cite{Murai2022}. In the diagonal setting, radicality depends both on the shape of $\lambda$ and on the number of sets of variables  $m$.

\subsection{Dependence on the number of sets of variables}
To compare $I_\lambda$ and $V_\lambda$ for different $m$, we need maps that add or forget sets of variables.

\begin{definition}\label{def:maps}
Let $m'<m$, $\mathbf{v}=(v_{m'+1},\dots,v_m)^T\in\C^{m-m'}$ and consider the maps
\begin{equation*}
\begin{aligned}
\ev_\mathbf{v}:\mathcal{R}_{m,n}&\longrightarrow\mathcal{R}_{m',n}, &\quad 
\iota_{\mathbf{v}}:\C^{m'}&\longrightarrow\C^m,\\
x_{r,i}&\longmapsto\begin{cases}x_{r,i}, & r\leq m',\\ v_r, & r>m',\end{cases} &\quad
\mathbf{c}&\longmapsto(c_1,\dots,c_{m'},v_{m'+1},\dots,v_m)^T.
\end{aligned}
\end{equation*}
Extend $\iota_{\mathbf{v}}$ to points $P=(\mathbf{u}_{\bullet,1},\dots,\mathbf{u}_{\bullet,n})
\in\C^{m'n}$ by 
$\iota_{\mathbf{v}}(P):=(\iota_{\mathbf{v}}(\mathbf{u}_{\bullet,1}),\dots,\iota_{\mathbf{v}}(\mathbf{u}_{\bullet,n}))\in\C^{mn}$.
Note that $p(\iota_{\mathbf{v}}(P))=\ev_{\mathbf{v}}(p)(P)$ for all $p\in\mathcal{R}_{m,n}$ and 
$P\in\C^{m'n}$. Define the truncation $\rho:\C^{m}\to\C^{m'}$, $(c_1,\dots,c_m)^T\mapsto(c_1,\dots,c_{m'})^T$, and extend it to points $P=(\mathbf{u}_{\bullet,1},\dots,\mathbf{u}_{\bullet,n})\in\C^{mn}$  by $\rho(P):=(\rho(\mathbf{u}_{\bullet,1}),\dots,\rho(\mathbf{u}_{\bullet,n}))\in\C^{m'n}$. Note that $\rho\circ\iota_{\mathbf{v}}=\id_{\C^{m'n}}$, and that for all $f\in\mathcal{R}_{m',n}$ and $P\in\C^{mn}$, if we regard $f$ as a polynomial in $\mathcal{R}_{m,n}$, $\mathcal{R}_{m',n}\subseteq\mathcal{R}_{m,n}$,  $f(P)=f(\rho(P))$.

\end{definition}

When the number of sets of variables needs to be emphasized, we write $H_\mu^{(m)}$ for the set of points in $\C^{mn}$ of type $\mu$, $I_\lambda^{(m)}$ for the Specht ideal $I_\lambda\subseteq\mathcal{R}_{m,n}$ and $V_\lambda^{(m)}\subseteq\C^{mn}$ for its zero set. 
\begin{lemma}\label{lem:intersections}
Let $m'<m$, then
\begin{enumerate}
     \item $I_\lambda^{(m')}=I_\lambda^{(m)}\cap\mathcal{R}_{m',n}$,
     \item $I(V_\lambda^{(m')})=I(V_\lambda^{(m)})\cap\mathcal{R}_{m',n}$.
\end{enumerate}
\end{lemma}
\begin{proof}
\begin{enumerate}
    \item On the one hand, since $\mathcal{R}_{m',n}\subseteq\mathcal{R}_{m,n}$, by Proposition~\ref{prop:genIlambda}, we have
\begin{equation*}
I_\lambda^{(m')}=\langle\varepsilon_T(p)\mid p\in\Mon(\mathcal{R}_{m',n}),\,T\in\Tab(\lambda)\rangle \subseteq\langle\varepsilon_T(p)\mid p\in\Mon(\mathcal{R}_{m,n}),\,T\in\Tab(\lambda)\rangle=I_\lambda^{(m)},
\end{equation*}
and $I_\lambda^{(m')}\subseteq\mathcal{R}_{m',n}$ by definition. On the other hand, for all $\sigma\in S_n$, $p\in\mathcal{R}_{m,n}$, $\ev_{\mathbf{0}}(\sigma\cdot p)=\sigma\cdot\ev_{\mathbf{0}}(p)$, and $\ev_{\mathbf{0}}(\varepsilon_T(p))=\varepsilon_T(\ev_{\mathbf{0}}(p))$, then by Proposition~\ref{prop:genIlambda}
\begin{equation*}
\ev_{\mathbf{0}}(I_\lambda^{(m)})=\langle\varepsilon_T(\ev_{\mathbf{0}}(p))\mid p\in\Mon(\mathcal{R}_{m,n}),\,T\in\Tab(\lambda)\rangle \subseteq\langle\varepsilon_T(p)\mid p\in\Mon(\mathcal{R}_{m',n}),\,T\in\Tab(\lambda)\rangle=I_\lambda^{(m')}.
\end{equation*}
Let $f\in I_\lambda^{(m)}\cap\mathcal{R}_{m',n}$, since $f\in\mathcal{R}_{m',n}$, we have $\ev_{\mathbf{0}}(f)=f$, therefore $f\in\ev_{\mathbf{0}}(I_\lambda^{(m)})\subseteq I_\lambda^{(m')}$.
    \item Let $f\in I(V_\lambda^{(m')})$, that is $f(P')=0$ for all $P'\in H_\mu^{(m')}$ with $\mu\not\trianglelefteq\lambda$. Let $P\in H_\nu^{(m)}$ with $\nu\not\trianglelefteq\lambda$, and consider the truncated point $\rho(P)$, then $\rho(P)\in H_\mu^{(m')}$ with $\nu\prec\mu$, hence $\nu\trianglelefteq\mu$, then $\mu\not\trianglelefteq\lambda$, and $f(P)=f(\rho(P))=0$ since $f\in\mathcal{R}_{m',n}$, so $f\in I(V_\lambda^{(m)})$.
    
    On the other hand, since $P\in H_\mu^{(m')}$ if and only if $\iota_{\mathbf{0}}(P)\in H_\mu^{(m)}$, and by Theorem~\ref{thm:charzeros}, we have
    \begin{equation*}
     \iota_{\mathbf{0}}(V_\lambda^{(m')})\subseteq V_\lambda^{(m)}\subseteq \C^{mn}.
    \end{equation*}
    Let $f\in I(V_\lambda^{(m)})\cap\mathcal{R}_{m',n}$, $P'\in V_\lambda^{(m')}$, since $f\in \mathcal{R}_{m',n}$, $f(P')=f(\iota_{\mathbf{0}}(P'))=0$, and $\iota_{\mathbf{0}}(P')\in V_\lambda^{(m)}$, therefore $f\in I(V_\lambda^{(m')})$.
\end{enumerate}
\end{proof}
\begin{prop}\label{prop:bound}
Let $\lambda\vdash n$.
\begin{enumerate}
    \item If $I_\lambda\subseteq\mathcal{R}_{m,n}$ is radical, then $I_\lambda\subseteq\mathcal{R}_{m',n}$ is radical for all $m'<m$.
    \item If $I_\lambda\subseteq\mathcal{R}_{m,n}$ is not radical, then $I_\lambda\subseteq\mathcal{R}_{m',n}$ is not radical for all $m'>m$.
    \end{enumerate}
\end{prop}
\begin{proof}
If $I_\lambda^{(m)}$ is radical, we have $I_\lambda^{(m)}=I(V_\lambda^{(m)})$. Intersecting with $\mathcal{R}_{m',n}$ on both sides we get $I_\lambda^{(m)}\cap\mathcal{R}_{m',n}=I(V_\lambda^{(m)})\cap\mathcal{R}_{m',n}$, and by Lemma~\ref{lem:intersections}, $I_\lambda^{(m')}$ is radical. The second statement is the contrapositive of the first.
\end{proof}

\begin{definition}\label{def:threshold}
Let $\lambda\vdash n$, define
\begin{equation*}
m(\lambda):=\min\bigl\{m\geq1\;\bigm|\;I_\lambda\subseteq\mathcal{R}_{m,n}\text{ is not radical}\bigr\}
\end{equation*}
with $m(\lambda)=\infty$ when $I_\lambda\subseteq\mathcal{R}_{m,n}$ is radical for every $m$.
\end{definition}

We now show that it is enough to check radicality for $m=n$. The key point is that highest weight vectors involve at most $n$ sets of variables.
\begin{lemma}\label{lem:Rnn}
Let $g\in\mathcal{R}_{m,n}$ be a highest weight vector of weight $\nu$. Then $g$ is multihomogeneous of multidegree $\nu$, with $\nu\vdash d$, $\len(\nu)\leq\min(n,m)$, and $g\in\mathcal{R}_{\len(\nu),n}$. In particular, if $n\leq m$, $g\in\mathcal{R}_{n,n}$.
\end{lemma}
\begin{proof}
By Theorem~\ref{thm:cauchy}, each graded piece $(\mathcal{R}_{m,n})_d$ is completely reducible as a $\GL_m$-module, and  its decomposition is
\begin{equation*}
(\mathcal{R}_{m,n})_d \cong \bigoplus_{\substack{\nu \vdash d\\
\len(\nu)\leq\min(m,n)}}\mathbb{S}_\nu(\C^m)^{\oplus c_\nu}.
\end{equation*}
A highest weight vector $g$ of $(\mathcal{R}_{m,n})_d$ generates a copy of $\mathbb{S}_\nu(\C^m)$, the weight of $g$ is $\nu\vdash d$ with $\len(\nu)\leq\min(n,m)$ and it is multihomogeneous of multidegree $\nu$. Since $\nu_r=0$ for $r>\len(\nu)$, $g\in\mathcal{R}_{\len(\nu),n}$. 
\end{proof}
\begin{thm}\label{thm:radicalinRnn}
If $I_\lambda\subseteq\mathcal R_{n,n}$ is radical, then $I_\lambda\subseteq\mathcal R_{m,n}$ is radical for every $m$.
\end{thm}
\begin{proof}
For $m<n$, apply Proposition~\ref{prop:bound}. Assume $I_\lambda^{(n)}$ is radical and let $m>n$. Since $I_\lambda\subseteq I(V_\lambda)$, it is enough to prove $I(V_\lambda)\subseteq I_\lambda$. Since $I(V_\lambda)\subseteq\mathcal{R}_{m,n}$ is $\GL_m$-stable and graded by total degree by Corollary~\ref{cor:GLminvariantvariety}, each $(I(V_\lambda))_d$ is generated by highest weight vectors by Proposition~\ref{prop:GLmfacts}(3). Let $g$ be a highest weight vector in $I(V_\lambda)$, then by Lemma~\ref{lem:Rnn}, $g\in\mathcal{R}_{n,n}$. By Lemma~\ref{lem:intersections}, and by radicality of $I^{(n)}_\lambda$
\begin{equation*}
g\in I(V_\lambda^{(m)})\cap\mathcal{R}_{n,n}=I(V_\lambda^{(n)})=I_\lambda^{(n)}\subseteq I_\lambda^{(m)},
\end{equation*}
and since $I(V_\lambda)$ is generated, as a $\GL_m$-module, by highest weight vectors, and $I_\lambda^{(m)}$ is $\GL_m$-stable by Lemma~\ref{lem:GLminv}, $I(V_\lambda^{(m)})\subseteq I_\lambda^{(m)}$, and $I_\lambda\subseteq\mathcal{R}_{m,n}$ is radical.
\end{proof}
\subsection{Radicality for hook partitions}
We begin with the three smallest hook families.  Their varieties have particularly simple collision descriptions, and this makes it possible to compare the ideals directly with their radical ideals through explicit Gr\"obner bases.

The ideal $I_{(n)}$ is radical for all $m$, since $1$ is $S_n$-stable, hence $1\in\mathcal{R}_{m,n,(n)}$. $V_{(n)}=\emptyset$ by Theorem~\ref{thm:charzeros}, and $I_{(n)}=\mathcal{R}_{m,n}$, which is radical. We will prove that $I_{(n-1,1)}$ and $I_{(n-2,1,1)}$ are radical in Theorems~\ref{thm:Ia1} and~\ref{thm:a11}. 

Let $<$ be a monomial order on a polynomial ring $\mathcal{R}$. We denote by $\ini_<(f)$ (or $\ini(f)$ if the order is implicitly understood) the initial monomial of $f\in \mathcal{R}$ with respect to $<$. Let $I$ be an ideal of $\mathcal{R}$. A finite subset $G\subseteq I$ is called a Gröbner basis of $I$ with respect to $<$ if $\langle \ini_<(f)\, |\, 0 \neq f\in I\rangle=\langle\ini_<(g),\,g\in G\rangle$. See \cite[\S15]{Eisenbud1995Commutative} for results on Gröbner bases. We will find Gröbner bases for the ideals $I_{(n-1,1)}$ and  $I_{(n-2,1,1)}$ in $\mathcal{R}_{m,n}$ with respect to the monomial order of Definition~\ref{def:monomialordergen}.
\begin{definition}\label{def:orderexp}
Let $\boldsymbol{\alpha}=(\alpha_1,\dots,\alpha_m),\boldsymbol{\beta}=(\beta_1,\dots,\beta_m) \in\mathbb{N}^m$. For $1\leq s\leq m$, we write $|\boldsymbol{\alpha}|_s:=\sum_{r=1}^s\alpha_r$ for the \emph{partial sum} of $\boldsymbol{\alpha}$ up to index $s$, and $|\boldsymbol{\alpha}|:=|\boldsymbol{\alpha}|_m$. We say $\boldsymbol{\beta}<\boldsymbol{\alpha}$ if there exists $t$, $1\leq t\leq m$, such that $|\boldsymbol{\alpha}|_s=|\boldsymbol{\beta}|_s$ for all $s>t$ and $|\boldsymbol{\beta}|_t<|\boldsymbol{\alpha}|_t$.
\end{definition}
\begin{example}
Consider $\boldsymbol{\alpha}=(1,2,1)$ and $\boldsymbol{\beta}=(1,1,2)$ in $\mathbb{N}^3$. Their partial sums are $|\boldsymbol{\alpha}|_3=|\boldsymbol{\beta}|_3=4$, $|\boldsymbol{\beta}|_2=2<3=|\boldsymbol{\alpha}|_2$, so $\boldsymbol{\beta}<\boldsymbol{\alpha}$.
\end{example}
We introduce new notation. Let $p=\prod_{r=1}^m\prod_{i=1}^n x_{r,i}^{\alpha_{r,i}}$ be a monomial in $\mathcal{R}_{m,n}$. For $1\leq i\leq n$, we write $\mathbf{x}_{\bullet,i}:=(x_{1,i},\dots,x_{m,i})^T$ for the vector of $i$-th variables across all sets. For $\boldsymbol{\alpha}_i=(\alpha_{1,i},\dots,\alpha_{m,i})^T\in\mathbb{N}^m$ we write $\mathbf{x}_{\bullet,i}^{\boldsymbol{\alpha}_i}:=x_{1,i}^{\alpha_{1,i}}\cdots x_{m,i}^{\alpha_{m,i}}$, and $p$ can be written as $p=\prod_{i=1}^n\mathbf{x}_{\bullet,i}^{\boldsymbol{\alpha}_i}$. We call $\boldsymbol{\alpha}_i$ the $i$-\emph{multiexponent} of $p$. Note that $\deg(p)=\sum_{i=1}^n|\boldsymbol{\alpha}_i|$.
\begin{definition}\label{def:monomialordergen}
Let $p=\prod_{i=1}^n\mathbf{x}_{\bullet,i}^{\boldsymbol{\alpha}_i}$ and $p'=\prod_{i=1}^n\mathbf{x}_{\bullet,i}^{\boldsymbol{\beta}_i}$ be monomials in $\mathcal{R}_{m,n}$, with $i$-multiexponents $\boldsymbol{\alpha}_i=(\alpha_{1,i},\dots,\alpha_{m,i})$ and $\boldsymbol{\beta}_i=(\beta_{1,i},\dots,\beta_{m,i})$, $1\leq i\leq n$, respectively. We say $p'<p$ if there exists $k$, $1\leq k\leq n$, such that $\boldsymbol{\alpha}_l=\boldsymbol{\beta}_l$ for all $l$ with $k<l\leq n$, and $\boldsymbol{\beta}_k<\boldsymbol{\alpha}_k$.
\end{definition}

Intuitively, we compare the $i$-multiexponents starting from $i=n$ down to $i=1$. When they first differ, we compare the total degree in all $m$ variables, then in the first $m-1$, and so on down to the degree in $x_{1,i}$ alone. Note that Definition~\ref{def:monomialordergen} indeed defines a monomial order: it is a block order, therefore it is total, compatible with multiplication and $1$ is minimal. In particular, $x_{1,i}>x_{2,i}>\dots>x_{m,i}$ for every $i$.
\begin{example}\label{ex:monomialorder}
Consider the polynomial ring $\mathcal{R}_{3,2}:=\C[\mathbf{x},\mathbf{y},\mathbf{z}]$, where $\mathbf{x}=(x_1,x_2)$, $\mathbf{y}=(y_1,y_2)$, $\mathbf{z}=(z_1,z_2)$. Let $p=x_1y_1^2z_1x_2y_2^2$, $p'=x_1y_1z_1^2x_2y_2^2$, with respective $i$-multiexponents for $i=1,2$
\begin{equation*}
\boldsymbol{\alpha}_1=(1,2,1), \quad \boldsymbol{\alpha}_2=(1,2,0),\qquad\text{ and }\qquad\boldsymbol{\beta}_1=(1,1,2),\quad \boldsymbol{\beta}_2=(1,2,0).
\end{equation*}
We compare the $2$-multiexponents first, and since $\boldsymbol{\alpha}_2=\boldsymbol{\beta}_2=(1,2,0)$, we compare the $1$-multiexponents, and since $\boldsymbol{\beta}_1<\boldsymbol{\alpha}_1$, $p'<p$.
\end{example}
From now on, $<$ will refer to the monomial order of Definition~\ref{def:monomialordergen}, where the ambient polynomial ring will be clear from context or specified explicitly.
The order $<$ is compatible with increasing the number $m$ of sets of variables. This lets us transfer Gröbner basis computations from fewer to more sets of variables.
\begin{lemma}\label{lem:iota}
Let $l<m$, $1\leq r_1< \cdots< r_l\leq m$, and consider the injective map 
\begin{equation*}
\iota:\mathcal {R}_{l,n}\longrightarrow\mathcal {R}_{m,n},\qquad x_{k,i}\longmapsto x_{r_k,i}
\end{equation*}
then
\begin{enumerate}
    \item $\iota(\ini(f))=\ini(\iota(f))$ for all $f\in\mathcal{R}_{l,n}$, and
    \item $\iota(\sigma\cdot f)=\sigma\cdot\iota(f)$ for all $\sigma\in S_n$, $f\in\mathcal{R}_{l,n}$. Consequently, $\iota(\varepsilon_T\cdot f)=\varepsilon_T\cdot\iota(f)$ for all $T\in\Tab(\lambda)$ for any $\lambda$, and for all $f\in\mathcal{R}_{l,n}$.
\end{enumerate}
\end{lemma}
\begin{proof}
\begin{enumerate}
    \item It suffices to show that $\iota$ preserves the order $<$. Let $p=\prod_{i=1}^n\mathbf{x}_{\bullet,i}^{\balpha_i}$ and $p'=\prod_{i=1}^n\mathbf{x}_{\bullet,i}^{\bbeta_i}$ be monomials in $\mathcal{R}_{l,n}$, with $i$-multiexponents $\balpha_i,\bbeta_i\in\N^l$, and suppose $p'<p$. Then $\iota(p)=\prod_{i=1}^n\mathbf{x}_{\bullet,i}^{\balpha_i'}$ and $\iota(p')=\prod_{i=1}^n\mathbf{x}_{\bullet,i}^{\bbeta_i'}$, where $\balpha_i',\bbeta_i'\in\N^m$ have the entries of $\balpha_i,\bbeta_i$ in positions $r_1<\dots<r_l$ and zeros elsewhere. If $p'<p$, let $k$ be the largest index such that $\boldsymbol{\alpha}_k\neq\boldsymbol{\beta}_k$, and let $t$ be such that $|\boldsymbol{\alpha}_k|_s=|\boldsymbol{\beta}_k|_s$ for $s>t$ but $|\boldsymbol{\beta}_k|_t<|\boldsymbol{\alpha}_k|_t$. Then $k$ is the largest index such that $\boldsymbol{\alpha'}_k\neq\boldsymbol{\beta'}_k$ and:
    \begin{itemize}
        \item if $t=l$, then we have $|\boldsymbol{\beta}'_k|_m<|\boldsymbol{\alpha}'_k|_m$ and $\iota(p')<\iota(p)$;
        \item if $t<l$, let $t'=r_{t+1}-1$, we have $|\boldsymbol{\beta}'_k|_{t'}<|\boldsymbol{\alpha}'_k|_{t'}$ and $\iota(p')<\iota(p)$.
    \end{itemize}
    Therefore we have proved that $\iota(\ini(f))=\ini(\iota(f))$.
    \item We have $\iota(\sigma\cdot x_{k,i})=\iota(x_{k,\sigma(i)})=x_{r_k,\sigma(i)}=\sigma\cdot x_{r_k,i}=\sigma\cdot\iota(x_{k,i})$, which extends to a general polynomial $f$. Extending by linearity $\iota(\varepsilon_T\cdot f)=\varepsilon_T\cdot\iota(f)$ for all $T\in\Tab(\lambda)$ for any $\lambda$, and for all $f\in\mathcal{R}_{l,n}$.
\end{enumerate}
\end{proof}
As a first case, we consider the hook partitions $(n-1,1)$. Their varieties are easy to describe: they consist exactly of the points whose $n$ columns are all equal, equivalently whose stabilizer is $S_n$.

\begin{thm}\label{thm:Ia1}
Let $\mathcal{B}=\{x_{r,i}-x_{r,1}\mid r=1,\dots,m,\,i=2,\dots,n\}$. Then
\begin{equation*}
I_{(n-1,1)}=I(V_{(n-1,1)})=\langle\mathcal{B}\rangle,
\end{equation*}
in particular $I_{(n-1,1)}$ is radical, and $\mathcal{B}$ is a Gr\"obner basis of $I_{(n-1,1)}$ with respect to $<$.
\end{thm}
\begin{proof}
By Theorem~\ref{thm:charzeros}, $V_{(n-1,1)}=H_{(n)}$, and
\begin{equation*}
I(V_{(n-1,1)})=\langle x_{r,i}-x_{r,1}\,|\,r=1,\dots,m,\,i=2,\dots,n\rangle.
\end{equation*}
For $2\leq i\leq n$, let $T_i=\mbox{\tiny$\begin{ytableau}1&\cdots&\\i\end{ytableau}$}\in\SYT((n-1,1))$.
Since $x_{r,i}$ is fixed by $R(T_i)$ and $C(T_i)=\{\id,(1\,i)\}$, we have $\varepsilon_{T_i}(x_{r,i})=|R(T_i)|\,(x_{r,i}-x_{r,1})$, and by Proposition~\ref{prop:etinvl}
\begin{equation*}
I(V_{(n-1,1)})=\langle \varepsilon_{T_i}(x_{r,i})\,|\,r=1,\dots,m,\,i=2,\dots,n\rangle\subseteq I_{(n-1,1)}\subseteq I(V_{(n-1,1)}),
\end{equation*}
and $I_{(n-1,1)}$ is radical. Finally, $\ini(x_{r,i}-x_{r,1})=x_{r,i}$, and these initial monomials are pairwise coprime, so all $S$-polynomials of $\mathcal{B}$ reduce to $0$ and $\mathcal{B}$ is a Gr\"obner basis of $I_{(n-1,1)}$.
\end{proof}

We will find a Gröbner basis of $I(V_{(n-2,1,1)})\subseteq\mathcal{R}_{m,n}$ by induction on $n$. We begin with the base case $n=3$ for all $m\geq1$ in Lemma~\ref{lem:GBa=1}.

\begin{lemma}\label{lem:GBa=1}
Let $p_{r,ij}:=x_{r,i}-x_{r,j}$ for $1\leq r\leq m$, $1\leq i<j\leq 3$, let $A$ be the unique standard Young tableau of shape $(1,1,1)$, let $q_{r,s}:=\varepsilon_A(x_{s,2}x_{r,3})$, and let
\begin{equation*}
\mathcal{B}^{(m)}:=\left\{q_{r,s}\,\mid\,1\leq r<s\leq m\right\}\cup\left\{p_{r,12}\,p_{s,13}\,p_{t,23}\,\mid\,1\leq r\leq s\leq t\leq m\right\}\subseteq\mathcal{R}_{m,3}.
\end{equation*}
Then $\mathcal{B}^{(m)}$ is a Gröbner basis of $I(V_{(1,1,1)})\subseteq\mathcal{R}_{m,3}$ with respect to $<$. In particular, $\ini(I(V_{(1,1,1)}))$ is generated by the families:
\begin{enumerate}
    \item[(i)] $x_{r,2}x_{r,3}^2$, $1\leq r\leq m$;
    \item[(ii)] $x_{r,2}x_{r,3}x_{s,3}$, $1\leq r<s\leq m$;
    \item[(iii)] $x_{r,2}x_{s,3}x_{t,3}$, $1\leq r<s<t\leq m$;
    \item[(iv)] $x_{r,2}x_{s,3}^2$, $1\leq r<s\leq m$;
    \item[(v)] $x_{s,2}x_{r,3}$, $1\leq r<s\leq m$.
\end{enumerate}
\end{lemma}
The proof of Lemma~\ref{lem:GBa=1} relies on the following two computational results.
\begin{lemma}\label{lem:computationm=3}
With the notation of Lemma~\ref{lem:GBa=1}:
\begin{enumerate}
    \item $I(V_{(1,1,1)})^{(l)}=\langle\mathcal{B}^{(l)}\rangle$, for $1\leq l\leq 3$.
    \item In $\mathcal{R}_{l,3}$ for $1\leq l\leq 6$, for all $f,g\in\mathcal{B}^{(l)}$, $S(f,g)$ reduces to $0$ modulo $\mathcal{B}^{(l)}$. 
\end{enumerate}
\end{lemma}
\begin{proof}
    The assertions were verified by exact computations over $\Q$ in \texttt{SageMath}~10.6. The code implementing the polynomial rings, the monomial order of Definition~\ref{def:monomialordergen}, and the Gr\"obner basis computations is available at \url{https://github.com/uggtos/diagonal_specht}.
\end{proof}
\begin{proof}[Proof of Lemma~\ref{lem:GBa=1}]
By Corollary~\ref{cor:closure}, $I\left(V_{(1,1,1)}\right)=I\left(\overline{H_{(2,1)}}\right)$. Let $p_{r,ij}:=x_{r,i}-x_{r,j}$ and $I_{ij}:=\langle p_{r,ij},\,r=1,\dots,m\rangle$, so that
\begin{equation*}
    I(V_{(1,1,1)})=I_{12}\cap I_{13}\cap I_{23}.
\end{equation*}
We want to prove that 
\begin{equation*}
\mathcal{B}^{(m)}:=\left\{q_{r,s}\,\mid\,r<s\right\}\cup\left\{p_{r,12}p_{s,13}p_{t,23}\,\mid\, 1\leq r\leq s\leq t\leq m\right\}
\end{equation*}
is a Gröbner basis of $I_{12}\cap I_{13}\cap I_{23}$.

We will show first that
$\langle\mathcal{B}^{(m)}\rangle=I(V_{(1,1,1)})$ for all $m$, and then that $\mathcal{B}^{(m)}$ is a Gröbner basis of the ideal it generates.

First, $p_{r,12}p_{s,13}p_{t,23}\in I_{12}I_{13}I_{23}\subseteq I_{12}\cap I_{13}\cap I_{23}=I(V_{(1,1,1)})$,
and $\varepsilon_A(x_{s,2}x_{r,3})\in I_{(1,1,1)}\subseteq I(V_{(1,1,1)})$ by Proposition~\ref{prop:etinvl}.

To prove $I(V_{(1,1,1)})\subseteq \langle\mathcal{B}^{(m)}\rangle$, we show first that $\langle\mathcal{B}^{(m)}\rangle$ is $\GL_m$-stable.
Let $g=(g_{sr})\in\GL_m$, $p_{r,ij}\in I_{ij}$, then $g\cdot p_{r,ij}:=\sum_{s=1}^m g_{sr}p_{s,ij}$, so $I_{ij}$ are $\GL_m$-stable, and their product is $\GL_m$-stable. 
Expanding $\varepsilon_A\cdot x_{s,2}x_{r,3}$ and using $p_{\cdot,13}=p_{\cdot,12}+p_{\cdot,23}$, we get
\begin{equation*}
q_{r,s}=p_{s,12}p_{r,13}-p_{r,12}p_{s,13}=p_{s,12}p_{r,23}-p_{r,12}p_{s,23}=p_{s,13}p_{r,23}-p_{r,13}p_{s,23},
\end{equation*}
in particular $q_{r,r}=0$, $q_{r,s}=-q_{s,r}$. For any $r,s,t$, we can write
\begin{eqnarray*}
p_{s,12}p_{r,13}p_{t,23}-p_{r,12}p_{s,13}p_{t,23}=q_{r,s}\,p_{t,23},\\
p_{t,12}p_{s,13}p_{r,23}-p_{r,12}p_{s,13}p_{t,23}=q_{r,t}\,p_{s,13},\\
p_{r,12}p_{t,13}p_{s,23}-p_{r,12}p_{s,13}p_{t,23}=q_{s,t}\,p_{r,12},
\end{eqnarray*}
so that exchanging any two of the indices $r,s,t$ of the product $p_{r,12}p_{s,13}p_{t,23}$ changes it by a multiple of $q_{r,s}$, $q_{r,t}$ or $q_{s,t}$, which lie in $\langle\mathcal{B}^{(m)}\rangle$ since $q_{r,r}=0$ and $q_{r,s}=-q_{s,r}$. Therefore $p_{r,12}p_{s,13}p_{t,23}\in\langle\mathcal{B}^{(m)}\rangle$ for all $r,s,t$.
Since for $g=(g_{sr})\in\GL_m$, $g\cdot p_{r,ij}:=\sum_{s=1}^m g_{sr}p_{s,ij}$, we have
\begin{align*}
    g\cdot q_{r,s}&=g\cdot(p_{s,12}p_{r,13}-p_{r,12}p_{s,13})
    =\sum_{k=1}^m\sum_{l=1}^m g_{ks}\,g_{lr}\,p_{k,12}\,p_{l,13}-\sum_{k=1}^m\sum_{l=1}^m g_{kr}\,g_{ls}\,p_{k,12}\,p_{l,13}\\
    &=\sum_{k=1}^m\sum_{l=1}^m g_{ks}\,g_{lr}\,\bigl(p_{k,12}\,p_{l,13}-p_{l,12}\,p_{k,13}\bigr)
    =\sum_{k=1}^m\sum_{l=1}^m g_{ks}g_{lr}\cdot q_{l,k}.
\end{align*}
Since $q_{l,l}=0$ and $q_{k,l}=-q_{l,k}$, this is a linear combination of the $q_{l,k}$ with $l<k$. Therefore $\langle\mathcal{B}^{(m)}\rangle$ is $\GL_m$-stable.

For $m\leq3$, $I(V_{(1,1,1)})=\langle\mathcal{B}^{(m)}\rangle$ by Lemma~\ref{lem:computationm=3}(1), so let $m>3$. Each $I(V_{(1,1,1)})_d$ is a $\GL_m$-submodule of $(\mathcal{R}_{m,3})_d$, generated by highest weight vectors by Proposition~\ref{prop:GLmfacts}(3), and if $h$ is a highest weight vector, $h\in\mathcal{R}_{3,3}$ by Lemma~\ref{lem:Rnn}. Then by Lemma~\ref{lem:intersections}, and Lemma~\ref{lem:computationm=3}(1),
\begin{equation*}
h\in I(V^{(m)}_{(1,1,1)})\cap\mathcal{R}_{3,3}=I(V^{(3)}_{(1,1,1)})=\langle\mathcal{B}^{(3)}\rangle
\subseteq\langle\mathcal{B}^{(m)}\rangle,
\end{equation*}
and since $\langle\mathcal{B}^{(m)}\rangle$ is $\GL_m$-stable, $I(V_{(1,1,1)})\subseteq\langle\mathcal{B}^{(m)}\rangle$.

Now we want to show that $\mathcal{B}^{(m)}$ is a Gröbner basis of $\langle\mathcal{B}^{(m)}\rangle$. By Buchberger's criterion, this is equivalent to showing that for all $f,g\in\mathcal{B}^{(m)}$, $S(f,g)$ reduces to $0$ modulo $ \mathcal{B}^{(m)}$. Every element of $\mathcal{B}^{(m)}$ involves at most three of the sets of variables $\mathbf{x}_1,\dots,\mathbf{x}_m$, so given $f,g\in\mathcal{B}^{(m)}$ there are indices $u_1<\cdots<u_l$ with $l\leq6$ such that $f,g\in\C[\mathbf{x}_{u_1},\dots,\mathbf{x}_{u_l}]$. Let $\iota:\mathcal{R}_{l,3}\longrightarrow\mathcal{R}_{m,3}$ be the map of Lemma~\ref{lem:iota}. We have
\begin{equation*}
\iota(p_{k,ij})=p_{u_k,ij},\qquad 
\iota(q_{r,s})=\iota\bigl(\varepsilon_A(x_{s,2}x_{r,3})\bigr)=\varepsilon_A\bigl(\iota(x_{s,2}x_{r,3})\bigr)=\varepsilon_A\bigl(x_{u_s,2}x_{u_r,3}\bigr)=q_{u_r,u_s},
\end{equation*}
where $\iota\circ\varepsilon_A=\varepsilon_A\circ\iota$ holds by Lemma~\ref{lem:iota}(2). Since $r<s$ implies $u_r<u_s$ and $r\leq s\leq t$ implies $u_r\leq u_s\leq u_t$, we have $\iota(\mathcal{B}^{(l)})\subseteq\mathcal{B}^{(m)}$, and $f=\iota(f')$, $g=\iota(g')$ for some $f',g'\in\mathcal{B}^{(l)}$. By Lemma~\ref{lem:iota}(1), $\iota$ maps initial terms to initial terms, so $\iota(S(f',g'))=S(f,g)$, and $\iota$ maps every step of the division of $S(f',g')$ by $\mathcal{B}^{(l)}$ in $\mathcal{R}_{l,3}$ to a step of the division of $S(f,g)$ by $\iota(\mathcal{B}^{(l)})\subseteq\mathcal{B}^{(m)}$ in $\mathcal{R}_{m,3}$. By Lemma~\ref{lem:computationm=3}(2), $S(f',g')$ reduces to $0$ modulo $\mathcal{B}^{(l)}$, therefore $S(f,g)$ reduces to $0$ modulo $\mathcal{B}^{(m)}$, and $\mathcal{B}^{(m)}$ is a Gröbner basis of $\langle\mathcal{B}^{(m)}\rangle=I(V_{(1,1,1)})$.

Since $\ini(q_{r,s})=x_{s,2}x_{r,3}$ and $\ini\left(p_{r,12}p_{s,13}p_{t,23}\right)=x_{r,2}x_{s,3}x_{t,3}$, $\ini(I(V_{(1,1,1)}))$ is generated by the families (i)--(v).
\end{proof}
The result for $m=1$ for all $n\geq3$ is the following. 
\begin{prop}\label{prop:GBm=1}
The initial ideal of $I(V_{(n-2,1,1)})\subseteq \mathcal{R}_{1,n}$ is
\begin{equation*}
\ini(I(V_{(n-2,1,1)}))=\langle x_ix_j^2\mid 1<i<j\leq n\rangle.
\end{equation*}
\end{prop}
\begin{proof}
From~\cite[Theorem 4.2.1]{Arnethesis}, $\mathcal{B}=\{\spe_T,\,T\in\Tab((n-2,1,1))\}$, $\spe_T$ being the Specht polynomials, is a Gröbner basis of $I_{(n-2,1,1)}\subseteq\mathcal{R}_{1,n}$ with respect to the order of Definition~\ref{def:monomialordergen} for $m=1$ and for any field of characteristic $0$, and by~\cite{Murai2022}, $I_{(n-2,1,1)}$ is radical over any infinite field, that is $\rad(I_{(n-2,1,1)})=I_{(n-2,1,1)}$. Since $\C$ is algebraically closed, the Nullstellensatz gives $\rad(I_{(n-2,1,1)})=I(V_{(n-2,1,1)})=I_{(n-2,1,1)}$, therefore $\mathcal{B}$ is a Gröbner basis of both $I_{(n-2,1,1)}$ and $I(V_{(n-2,1,1)})$.

For $T=\mbox{\tiny$\begin{ytableau}
        k &\cdots &\\
        l\\
        s
    \end{ytableau}$}$, $\spe_T=(x_k-x_l)(x_k-x_s)(x_l-x_s)$ has initial term $\pm x_ix_j^2$ with $i,j\in\{k,l,s\}$ and $1<i<j\leq n$, $
    \ini(\mathcal{B})=\langle x_ix_j^2,\,1<i<j\leq n\rangle$.
\end{proof}

We will use the following notation from~\cite{Murai2022}. For a partition $\eta=(\eta_1,\dots,\eta_\ell)$ 
and an integer $i>0$, we write $\eta+\langle i\rangle$ for the partition obtained by rearranging the sequence $(\eta_1,\dots,\eta_i+1,\dots,\eta_\ell)$, where  $\eta+\langle i\rangle=(\eta_1,\dots,\eta_\ell,1)$ 
when $i>\ell$.
\begin{remark}\label{rmk:lambda-}
Let $\lambda=(a,1^{n-a})$ with $n-a>0$, and let $\lambda^-=(a,1^{n-a-1})$ be the partition obtained by removing the last box of $\lambda$. 
Then by Lemma~\ref{lem:lenhook},
\begin{equation*}
\{\mu\vdash n\mid\mu\not\trianglelefteq\lambda\}=\{\mu\vdash n\mid\len(\mu)<n-a+1\}\text{ and } \{\eta\vdash n-1\mid\eta\not\trianglelefteq\lambda^-\}=
\{\eta\vdash n-1\mid\len(\eta)<n-a\}.
\end{equation*} 
Note that if $\eta\not\trianglelefteq\lambda^-$, then $\eta+\langle i\rangle\not\trianglelefteq\lambda$, since
$\len(\eta+\langle i\rangle)\leq\len(\eta)+1< n-a+1$.
\end{remark}

\begin{lemma}\label{lem:inductionstepge}
Let $\lambda=(a,1^{n-a})$ with $n-a>0$, and let $\lambda^-=(a,1^{n-a-1})$. Any polynomial $f\in I(V_\lambda)\subseteq\mathcal{R}_{m,n}$ can be written as
\begin{equation}
f= \sum_{\boldsymbol{\beta}\in\mathbb{N}^m}f_{\boldsymbol{\beta}}\,\mathbf{x}_{\bullet,n}^{\boldsymbol{\beta}},\quad \text{ with }\quad f_{\boldsymbol{\beta}}\in I(V_{\lambda^-})\subseteq\mathcal{R}_{m,n-1}.
\end{equation}
\end{lemma}
\begin{proof}
First, we write $f\in \mathcal{R}_{m,n}$ as a polynomial in the variables $\mathbf{x}_{\bullet,n}=(x_{1,n},\dots,x_{m,n})$ with coefficients in $\mathcal{R}_{m,n-1}$:
\begin{equation*}
f=\sum_{\boldsymbol{\beta}\in\mathbb{N}^m}f_{\boldsymbol{\beta}}\,\mathbf{x}_{\bullet,n}^{\boldsymbol{\beta}},\quad f_{\boldsymbol{\beta}}\in \mathcal{R}_{m,n-1}.
\end{equation*}
By Theorem~\ref{thm:charzeros}, $V_\lambda=\bigcup_{\mu\not\trianglelefteq\lambda}H_\mu$, $V_{\lambda^-}=\bigcup_{\eta\not\trianglelefteq\lambda^-}H_\eta$. 
We want to prove that $f_{\boldsymbol{\beta}}(P)=0$ for all $P\in V_{\lambda^-}\subseteq\C^{m(n-1)}$.
Let $P=(\mathbf{u}_{\bullet,1},\dots,\mathbf{u}_{\bullet,n-1})\in V_{\lambda^-}$, that is, $P\in H_\eta$ with $\eta\not\trianglelefteq \lambda^-$. Let
$\mathbf{a}\in\C^m$, and $Q:=(\mathbf{u}_{\bullet,1},\dots,\mathbf{u}_{\bullet,n-1},\mathbf{a})\in\C^{mn}$ be the point obtained by appending $\mathbf{a}$ to $P$ as the $n$-th column. Then 
$Q\in H_{\eta+\langle i\rangle}$ for some $i$. By 
Remark~\ref{rmk:lambda-}, $\eta+\langle i\rangle\not\trianglelefteq\lambda$, therefore $Q\in V_\lambda$. Since $f\in I(V_\lambda)$, we have
\begin{equation*}
f(Q)=\sum_{\boldsymbol{\beta}\in\mathbb{N}^m}f_{\boldsymbol{\beta}}(P)\,\mathbf{a}^{\boldsymbol{\beta}}=0.
\end{equation*}
Lastly, since $\mathbf{a}$ is chosen freely, $f_{\boldsymbol{\beta}}(P)=0$ for all 
$\boldsymbol{\beta}$.
\end{proof}
\begin{lemma}\label{lem:phi}
Let $g\in\mathcal{R}_{m,n}$ be multihomogeneous, and assume $x_{r,i}\nmid\ini(g)$ 
for $1<i\leq n$, $r>l$. 
Then $\ini(\ev_{\mathbf{1}}(g))=\ev_{\mathbf{1}}(\ini(g))$, $\ev_{\mathbf{1}}:\mathcal{R}_{m,n}\longrightarrow\mathcal{R}_{l,n}$ being the map from Definition~\ref{def:maps} with $m'=l$ where $\mathbf{1}=(1,\dots,1)\in\C^{m-l}$.
\end{lemma}
\begin{proof}
It suffices to show that if $\ev_{\mathbf{1}}(p)=\ev_{\mathbf{1}}(p')$, then $p=p'$, and $\ev_{\mathbf{1}}$ preserves 
the order $<$. Let $p:=\ini(g)$ with $i$-multiexponents $\boldsymbol{\alpha}_i$, $1\leq i\leq n$, 
and let $p'$ be another monomial of $g$ with $i$-multiexponents $\boldsymbol{\beta}_i$, $1\leq i\leq n$. 

Since $x_{r,i}\nmid p$ for $1<i\leq n$, $r>l$, 
we have $\alpha_{r,i}=0$ for $1<i\leq n$, $r>l$, and since $g$ is multihomogeneous, $\sum_{i=1}^n\beta_{r,i}=\sum_{i=1}^n\alpha_{r,i}$ for all $1\leq r\leq m$, 
in particular $\sum_{i=1}^n\beta_{r,i}=\alpha_{r,1}$ for all $r>l$.

Suppose $\ev_{\mathbf{1}}(p)=\ev_{\mathbf{1}}(p')$,
\begin{equation*}
\ev_{\mathbf{1}}(p)=\left(\prod_{r=1}^l\prod_{i=1}^n x_{r,i}^{\alpha_{r,i}}\right)
,\qquad
\ev_{\mathbf{1}}(p')=\left(\prod_{r=1}^l\prod_{i=1}^n x_{r,i}^{\beta_{r,i}}\right).
\end{equation*}
When $r\leq l$, $\alpha_{r,i}=\beta_{r,i}$ for all $1\leq i\leq n$. When $r>l$, assume $\beta_{r,i}>0$ 
for some $i>1$. We can assume that $i$ is maximal, and then $i$ is maximal such that $\boldsymbol{\alpha}_i\neq\boldsymbol{\beta}_i$, and since $\alpha_{r,i}=0$ for all $r>l$, $|\boldsymbol{\beta}_i|>|\boldsymbol{\alpha}_i|$ and $p<p'$, which contradicts $p=\ini(g)$, 
therefore $\beta_{r,i}=0$ for all $i>1$, $r>l$, and $\beta_{r,1}=\alpha_{r,1}$ by homogeneity. Therefore $p=p'$.

If $p'<p$, let $k$ be the largest index such that $\boldsymbol{\alpha}_k\neq\boldsymbol{\beta}_k$, 
and let $t$ be such that $|\boldsymbol{\alpha}_k|_s=|\boldsymbol{\beta}_k|_s$ for $s>t$ 
but $|\boldsymbol{\beta}_k|_t<|\boldsymbol{\alpha}_k|_t$.
\begin{enumerate}
    \item If $k>1$, $t\geq l$, then since $\alpha_{r,k}=0$ for $r>l$ 
$|\boldsymbol{\alpha}_k|_l=|\boldsymbol{\alpha}_k|_t>|\boldsymbol{\beta}_k|_t
\geq|\boldsymbol{\beta}_k|_l$, so $\ev_{\mathbf{1}}(p')<\ev_{\mathbf{1}}(p)$.
\item If $k>1$, $t<l$, then since $|\boldsymbol{\alpha}_k|_s=|\boldsymbol{\beta}_k|_s$ for $s>t$ 
and $|\boldsymbol{\alpha}_k|_t>|\boldsymbol{\beta}_k|_t$, we have 
$\ev_{\mathbf{1}}(p')<\ev_{\mathbf{1}}(p)$ by Definition~\ref{def:monomialordergen} applied to $\mathcal{R}_{l,n}$.
\item If $k=1$, $\boldsymbol{\alpha}_i=\boldsymbol{\beta}_i$ for all 
$i>1$, that is $\alpha_{r,i}=\beta_{r,i}$ for all $r$ and $i>1$. In 
particular, $\alpha_{r,i}=\beta_{r,i}=0$ for all $r>l$ and $i>1$. By 
multihomogeneity, $\alpha_{r,1}=\beta_{r,1}$ for all $r>l$, so the 
comparison of $\ev_{\mathbf{1}}(p')$ and $\ev_{\mathbf{1}}(p)$ reduces to their $\mathcal{R}_{l,n}$-parts 
at position $k=1$, and the same argument as case $k>1$ gives $\ev_{\mathbf{1}}(p')<\ev_{\mathbf{1}}(p)$.
\end{enumerate}

Therefore $\ev_{\mathbf{1}}$ preserves the order $<$, and $\ini(\ev_{\mathbf{1}}(g))=\ev_{\mathbf{1}}(\ini(g))$.
\end{proof}

\begin{prop}\label{prop:initial}
    $\ini(I(V_{(n-2,1,1)}))$ is contained in the ideal generated by the families:
    \begin{enumerate}
        \item[(i)] $x_{r,i}x_{r,j}^2$, $1\leq r\leq m$, $1<i<j\leq n$;
        \item[(ii)] $x_{r,i}x_{r,j}x_{s,j}$, $1\leq r<s\leq m$, $1<i<j\leq n$;
        \item[(iii)] $x_{r,i}x_{s,j}x_{t,j}$, $1\leq r<s<t\leq m$, $1<i<j\leq n$;
        \item[(iv)] $x_{r,i}x_{s,j}^2$, $1\leq r<s\leq m$, $1<i<j\leq n$;
        \item[(v)] $x_{s,i}x_{r,j}$, $1\leq r<s\leq m$, $1<i<j\leq n$.
    \end{enumerate}
\end{prop}
\begin{proof}
We prove the result by induction on $n$. The case $n=3$ is proven 
in Lemma~\ref{lem:GBa=1} for all $m$.
Let $f\in I(V_{(n-2,1,1)})\subseteq\mathcal{R}_{m,n}$. Since $I(V_{(n-2,1,1)})$ is multigraded by Corollary~\ref{cor:GLminvariantvariety}, we can assume that $f$ is multihomogeneous, and by Lemma~\ref{lem:inductionstepge}, we can write $f$ as
\begin{equation*}
f=\sum_{\boldsymbol{\beta}\in\mathbb{N}^m}
f_{\boldsymbol{\beta}}\mathbf{x}_{\bullet,n}^{\boldsymbol{\beta}},\quad f_{\boldsymbol{\beta}}\in I(V_{(n-2,1)})\subseteq\mathcal{R}_{m,n-1},
\end{equation*}
where each $f_{\boldsymbol{\beta}}$ is multihomogeneous, with $\ini(f)=\mathbf{x}_{\bullet,n}^{\boldsymbol{\alpha}}\ini(g)$, 
where $\boldsymbol{\alpha}=(\alpha_1,\dots,\alpha_m)\in\mathbb{N}^m$ is maximal 
in the order on $\mathbb{N}^m$ of Definition~\ref{def:orderexp}, 
and $g:=f_{\boldsymbol{\alpha}}\in\mathcal{R}_{m,n-1}$.

Since $g\in I(V_{(n-2,1)})\subseteq\mathcal{R}_{m,n-1}$, by Theorem~\ref{thm:Ia1} we have $x_{r,i}\mid\ini(g)$ for some $1\leq r\leq m$, 
$1<i\leq n-1$.
We refer to the five types of generators as \textrm{(i)}--\textrm{(v)}.
\begin{itemize}

\item Case $\alpha_l\geq 2$ for some $l$.
If $r<l$, then $x_{r,i}x_{l,n}^2\mid\ini(f)$, $x_{r,i}x_{l,n}^2$ is of type \textrm{(iv)}. 
If $r=l$, then $x_{l,i}x_{l,n}^2\mid\ini(f)$, $x_{l,i}x_{l,n}^2$ is of type \textrm{(i)}. 
If $r>l$, then $x_{r,i}x_{l,n}\mid\ini(f)$, $x_{r,i}x_{l,n}$ is of type \textrm{(v)}.

\item Case $\alpha_s=1$, $\alpha_t=1$ for some $s<t$.
If $r<s$, then $x_{r,i}x_{s,n}x_{t,n}\mid \ini(f)$, $x_{r,i}x_{s,n}x_{t,n}$ is of type \textrm{(iii)}.
If $r>s$, then $x_{r,i}x_{s,n}\mid \ini(f)$, $x_{r,i}x_{s,n}$ is of type \textrm{(v)}.
If $r=s$, then $x_{r,i}x_{r,n}x_{t,n}\mid \ini(f)$, and $x_{r,i}x_{r,n}x_{t,n}$ is of type \textrm{(ii)}.

\item Case $\alpha_l=1$, $\alpha_r=0$ for all $r\neq l$. Write
\begin{equation*}
f=\sum_{r=l}^mx_{r,n}g_r+g_0
\end{equation*}
where $g_0,g_r\in\mathcal{R}_{m,n-1}$, and  $f$ does not depend 
on $x_{r,n}$ for $r<l$. Since $\ini(f)=x_{l,n}\ini(g_l)$, if $x_{r,i}\mid \ini(g_l)$ for some $r>l$, $1<i\leq n-1$, $x_{r,i}x_{l,n}\mid\ini(f)$, which is of type \textrm{(v)}.

We want to prove that if $x_{r,i}\nmid \ini(g_l)$ for all $r>l$, then there is a monomial of type (i)--(v) in $\mathcal{R}_{l,n-1}$ that divides $\ini(f)$.
Let $\ev_{\mathbf{1}}: \mathcal{R}_{m,n-1}\to\mathcal{R}_{l,n-1}$ and $\iota_{\mathbf{1}}: \C^{l(n-1)}\to \C^{m(n-1)}$ be the maps from Definition~\ref{def:maps}. Let $P=(\mathbf{u}_{\bullet,1},\dots,\mathbf{u}_{\bullet,n-1})
\in V_{(n-3,1,1)}\subseteq\C^{l(n-1)}$. By Lemma~\ref{lem:lenhook} and 
Theorem~\ref{thm:charzeros}, the columns of $P$, $\mathbf{u}_{\bullet,i}$, can take at most 2 different values, $\mathbf{u}_{\bullet,i}\in
\{\mathbf{a},\mathbf{b}\}$ for some $\mathbf{a},\mathbf{b}\in\C^l$.

Consider the points $Q_1$ and $Q_2$ obtained by appending $\iota_{\mathbf{1}}(\mathbf{a})$ and $\iota_{\mathbf{1}}(\mathbf{b})$ to $\iota_{\mathbf{1}}(P)$ as the $n$-th column, respectively:
\begin{equation*}
Q_1:=(\iota_{\mathbf{1}}(P),\iota_{\mathbf{1}}(\mathbf{a})),\qquad Q_2:=(\iota_{\mathbf{1}}(P),\iota_{\mathbf{1}}(\mathbf{b})).
\end{equation*}
If $P=(\mathbf{a},\dots,\mathbf{a})\in H_{(n-1)}$, then choose any $\mathbf{b}\neq\mathbf{a}$. 
Since the columns of $Q_1$ and $Q_2$ can take at most 2 different values, $\iota_{\mathbf{1}}(\mathbf{u}_{\bullet,i}),\iota_{\mathbf{1}}(\mathbf{a}),\iota_{\mathbf{1}}(\mathbf{b})\in\{\iota_{\mathbf{1}}(\mathbf{a}),\iota_{\mathbf{1}}(\mathbf{b})\}$, 
by the same results as before we have $Q_1,Q_2\in V_{(n-2,1,1)}$.

Since $Q_1,Q_2\in V_{(n-2,1,1)}$ and $f\in I(V_{(n-2,1,1)})$, $Q_1$ and $Q_2$ are two roots of the degree-one polynomial in $x_{l,n}$
\begin{equation*}
h(x_{l,n}):=x_{l,n}(g_l\circ\iota_{\mathbf{1}})(P)+\sum_{r=l+1}^m(g_r\circ\iota_{\mathbf{1}})(P)+(g_0\circ\iota_{\mathbf{1}})(P),
\end{equation*}
$h(a_l)=f(Q_1)=0$ and $h(b_l)=f(Q_2)=0$. When $a_l\neq b_l$, this forces $h\equiv0$, so $(g_l\circ\iota_{\mathbf{1}})(P)=\ev_{\mathbf{1}}(g_l)(P)=0$. The locus $a_l\neq b_l$ (together with the constant points, where $\mathbf{b}$ can be chosen freely) is dense in every irreducible component of the variety, so $g_l\circ\iota_{\mathbf{1}}$ vanishes on a dense subset of each component and hence everywhere, giving $\ev_{\mathbf{1}}(g_l)\in I(V_{(n-3,1,1)})\subseteq \mathcal{R}_{l,n-1}$. By the induction hypothesis, there is a monomial of type (i)--(v) in $\mathcal{R}_{l,n-1}$ that divides $\ini(\ev_{\mathbf{1}}(g_l))$, and by Lemma~\ref{lem:phi}, $\ev_{\mathbf{1}}(\ini(g_l))=\ini(\ev_{\mathbf{1}}(g_l))$, so the same monomial divides $\ini(f)$.
\item Case $\boldsymbol{\alpha}=0$. In this case, $f\in\mathcal{R}_{m,n-1}$. Let $P=(\mathbf{u}_{\bullet,1},\dots,\mathbf{u}_{\bullet,n-1})\in V_{(n-3,1,1)}\subseteq\C^{m(n-1)}$, that is $\mathbf{u}_{\bullet,i}\in\{\mathbf{a},\mathbf{b}\}$ for some $\mathbf{a},\mathbf{b}\in\C^m$ by Lemma~\ref{lem:lenhook} and Theorem~\ref{thm:charzeros}. Consider the point $Q=(\mathbf{u}_{\bullet,1},\dots,\mathbf{u}_{\bullet,n-1},\mathbf{a})\in\C^{mn}$. Again by Lemma~\ref{lem:lenhook} and Theorem~\ref{thm:charzeros}, $Q\in V_{(n-2,1,1)}$ and $f(P)=f(Q)=0$, hence $f\in I(V_{(n-3,1,1)})$ and the induction hypothesis applies.
\end{itemize}
\end{proof}
The reverse inclusion holds as well (see the proof of Theorem~\ref{thm:a11}), so the families (i)--(v) in fact generate $\ini(I(V_{(n-2,1,1)}))$.
\begin{thm}\label{thm:a11}
$I_{(n-2,1,1)}\subseteq\mathcal{R}_{m,n}$ is radical.
\end{thm}
\begin{proof}
Since $I\subseteq\rad(I)$ for any ideal $I$, we want to show that $I(V_{(n-2,1,1)})\subseteq I_{(n-2,1,1)}$. For $1<i<j\leq n$, let $T_{ij}\in\Tab((n-2,1,1))$ be the tableau with first column $1,i,j$, and let $\mathcal{M}_{ij}$ be the set of monomials of type (i)--(v) of Proposition~\ref{prop:initial} with these $i$ and $j$.  For each $p \in \mathcal{M}_{ij}$, the row symmetrizer of $T_{ij}$ acts on $p$ by a nonzero scalar. In the expansion of the column antisymmetrizer, the identity permutation contributes $p$, whereas every nonidentity column permutation produces a monomial strictly smaller than $p$ for the order of Definition~\ref{def:monomialordergen}. Hence $\ini(\varepsilon_{T_{ij}}(p))=p$. Let $\mathcal{B}=\{\varepsilon_{T_{ij}}(p)\mid 1<i<j\leq n,\,p\in\mathcal{M}_{ij}\}$, then $\ini(\mathcal{B})$ is the set of
monomials (i)--(v) of Proposition~\ref{prop:initial}. 
By Proposition~\ref{prop:etinvl}, $\mathcal{B}\subseteq I_{(n-2,1,1)}\subseteq I(V_{(n-2,1,1)})$, so $\langle\ini(\mathcal{B})\rangle\subseteq\ini(I(V_{(n-2,1,1)}))$, and the reverse inclusion is Proposition~\ref{prop:initial}. Hence $\mathcal{B}$ is a Gröbner basis of $I(V_{(n-2,1,1)})$, so $\rad(I_{(n-2,1,1)})=I(V_{(n-2,1,1)})=\langle\mathcal{B}\rangle\subseteq I_{(n-2,1,1)}$, and $I_{(n-2,1,1)}$ is radical.
\end{proof}

The results of this section naturally suggest the following conjecture, which is verified for hooks of length at most 3.

\begin{conjecture}\label{conj:hook}
    If $\lambda$ is a hook partition, then $I_\lambda\subseteq\mathcal{R}_{m,n}$ is radical for all $n$ and $m$.
\end{conjecture}
\begin{remark}\label{rmk:haiman}
The conjecture holds when $\len(\lambda)\leq 3$ by Theorems~\ref{thm:Ia1} and~\ref{thm:a11}, and it holds for $\lambda=(1^n)$ and $m=2$, since $I_{(1^n)}\subseteq\C[\mathbf{x},\mathbf{y}]$ is the ideal generated by the $S_n$-alternating polynomials, and by Theorem~\ref{thm:charzeros} and Lemma~\ref{lem:alambdahooks}, $V_{(1^n)}=\{P\in\C^{2n}\,\mid\,\mathbf{u}_{\bullet,i}=\mathbf{u}_{\bullet,j}\text{ for some }1\leq i<j\leq n\}$, so \begin{equation*}
I(V_{(1^n)})=\bigcap_{i<j}\langle x_i-x_j,\,y_i-y_j\rangle.
\end{equation*}
Haiman \cite[Corollary~3.8.3, $d=1$]{Haiman2001} proved $I_{(1^n)}=I(V_{(1^n)})$. By Proposition~\ref{prop:bound}, $m((1^n))>2$ for all $n$. For $\lambda=(1^4)$, a computation for $m=4$ and Theorem~\ref{thm:radicalinRnn} give $m((1^4))=\infty$. Similarly, a computation for $m=5$ and Theorem~\ref{thm:radicalinRnn} give $m((2,1^3))=\infty$.
\end{remark}

\subsection{Non-radicality for non-hook partitions}\label{sec:nonradicality}
The preceding positive results suggest that hooks behave differently from non-hook partitions.  We now develop three obstructions to radicality.  They compare the first possible occurrences of different $S_n$-types by successively coarser invariants: content, multidegree, and total degree.

\subsubsection{Minimum degree in \texorpdfstring{$I_\lambda$}{I\_lambda}}
Our criteria for non-radicality compare the minimum degrees, multidegrees and contents of $I_\lambda$ and $I_\mu$ for $\mu\triangleleft\lambda$. We first compute the minimum degree of $I_\lambda$.
Given $T\in\Tab(\lambda)$, consider a tableau $E$ of shape $\lambda$,
with entries $E(b)=(E(b)_1,\dots,E(b)_m)\in\mathbb{N}^m$, that we call multiexponents. Then $T$ and $E$ define a monomial
\begin{equation*}
\mathbf{x}_T^E:=\prod_{b\in\lambda}\prod_{r=1}^{m}x_{r,T(b)}^{E(b)_r},
\end{equation*}
the $i$-multiexponent of $\mathbf{x}_T^E$ is $E(b)$ where $b$ is the box of $\lambda$ with entry $T(b)=i$, and the total degree is
$\deg(\mathbf{x}_T^E)=\sum_{b\in\lambda}|E(b)|$, where $|E(b)|=\sum_{r=1}^mE(b)_r$.

Let
$c_j=\binom{m+j-1}{j}$ be the number of monomials of degree $j$ in $m$ variables, and $N_d=\sum_{j=0}^dc_j$ be the number of monomials of degree at most $d$ in $m$ variables, with the convention $N_{-1}:=0$. For $b\geq0$, we call the rows $N_{b-1}+1,\dots,N_b$ of $\lambda$ the $b$-th \emph{band} of $\lambda$.
The \emph{minimal filling} of $\lambda$ (with respect to $m$)
is the tableau $E$ of shape $\lambda$ whose boxes in row $i$ all contain the
$i$-th smallest element of $\N^m$ for the order of
Definition~\ref{def:orderexp}. Note that the rows of the $b$-th band of
$\lambda$ are filled with multiexponents of degree $b$.
\begin{prop}\label{prop:degreegen}
Let $\lambda=(\lambda_1,\dots,\lambda_\ell)\vdash n$, let $d$ be such that $N_{d-1}< \ell \leq N_d$, then 
\begin{equation}\label{eq:minimdeg}
D_\lambda:=\min_{\substack{f\in I_\lambda\\ f\neq 0}}\{\deg(f)\}=\sum_{j=1}^dj\sum_{i=N_{j-1}+1}^{N_j}\lambda_i.
\end{equation}
\end{prop}
\begin{proof}
Fix $m\geq 1$. First, we prove existence. We find a nonzero polynomial of degree $D_\lambda$ that is in $I_\lambda$. Let $A$ be the superstandard tableau of shape $\lambda$ and $E$ be the minimal filling of $\lambda$, and consider $\mathbf{x}_A^E$, which has degree $D_\lambda$ (see Example~\ref{ex:mindeg}). We will show that $\varepsilon_A\cdot\mathbf{x}_A^E\neq0$. For all $\sigma\in R(A)$ we have $\sigma\cdot \mathbf{x}_A^E=\mathbf{x}_A^E$, 
\begin{equation*}    
\varepsilon_A\cdot\mathbf{x}_A^E=\alpha(A)\beta(A)\cdot\mathbf{x}_A^E=|R(A)|\alpha(A)\cdot \mathbf{x}_A^E.
\end{equation*}
Let $c$ be the coefficient of $\mathbf{x}_A^E$ in $\alpha(A)\cdot \mathbf{x}_A^E$. Since $\tau=\id$ is the only possibility for $\tau\cdot\mathbf{x}_A^E=\mathbf{x}_A^E$, $c=1$, $\varepsilon_A\cdot\mathbf{x}_A^E\neq0$. 

Now we want to prove minimality, that is, if $p\in\Mon(\mathcal{R}_{m,n})$, $\deg(p)<D_\lambda$, $S\in\Tab(\lambda)$, then $\alpha(S)\cdot p=0$. Let $F$ be the tableau of multiexponents such that $p=\mathbf{x}_S^F$, and decompose both $\mathbf{x}_S^F$ and $\mathbf{x}_A^E$ into columns,
\begin{equation*}
\mathbf{x}_S^F=\prod_{i}\mathbf{x}_{S_i}^{F_i},\quad\mathbf{x}_A^E=\prod_{i}\mathbf{x}_{A_i}^{E_i}.
\end{equation*}
Since $\sum_{i}\deg(\mathbf{x}_{S_i}^{F_i})=\deg(\mathbf{x}_S^F)<\deg(\mathbf{x}_A^E)=\sum_i \deg(\mathbf{x}_{A_i}^{E_i})$, there exist some $i$ for which $\deg(\mathbf{x}_{S_i}^{F_i})<\deg(\mathbf{x}_{A_i}^{E_i})$. But $E$ is filled with the smallest possible distinct multiexponents in each column, so if $\deg(\mathbf{x}_{S_i}^{F_i})<\deg(\mathbf{x}_{A_i}^{E_i})$, there must be at least one multiexponent repeated in $F_i$, therefore there exists $\pi\in C_i(S)$ such that $\pi\cdot p=p$, and as we have seen in the proof of Lemma~\ref{lem:projectionyoung}, this implies $\alpha(S)\cdot p=0$. In particular, $\varepsilon_S\cdot p=0$. Since this is true for any monomial of degree $<D_\lambda$, it is also true for any polynomial of degree $<D_\lambda$, and $D_\lambda$ is the minimum degree in $I_\lambda$.
\end{proof}

\begin{example}\label{ex:mindeg}
    When $m=2$, the multiexponents are $(0,0)<(0,1)<(1,0)<(0,2)<(1,1)<(2,0)<\cdots$, and the minimum degree of polynomials in $I_{(3,3,2,2)}$ is the one corresponding to the tableau of multiexponents
    \begin{equation*}
        \begin{ytableau}
            00 & 00 & 00 \\
            01 & 01 & 01\\
            10 & 10 \\
            02 & 02
        \end{ytableau}, \quad \min_{f\in I_{(3,3,2,2)}}\{\deg(f)\}=9=1\cdot3+1\cdot2+2\cdot2.
    \end{equation*}
\end{example}

\begin{remark}\label{rmk:IiffR}
Note that if $p\in\mathcal{R}_{m,n}$ and $D=\deg(p)$ is the minimum degree in $I_\lambda$, then $p\in I_\lambda$ if and only if $p\in\mathcal{R}_{m,n,\lambda}$. Since $I_\lambda=\langle\mathcal{R}_{m,n,\lambda}\rangle$, and $(\mathcal{R}_{m,n,\lambda})_d=0$ for $d<D$, the degree $D$ part of $I_\lambda$ comes only from the degree $D$ part of $\mathcal{R}_{m,n,\lambda}$, so $(I_\lambda)_D=(\mathcal{R}_{m,n,\lambda})_D$.
\end{remark}

\begin{corollary}\label{cor:posetideals}
Let $\lambda\vdash n$ and $I_\lambda\subseteq\mathcal{R}_{m,n}$. The map $\lambda\mapsto I_\lambda$ from $(\partn,\trianglelefteq)$ to $(\{I_\lambda\mid\lambda\in \partn\},\subseteq)$ is an isomorphism of posets if and only if $m=1$ or $n\leq 3$.
\end{corollary}
\begin{proof} For $m=1$, the result is proved in~\cite{moustrou2021symmetric}. When  $n\leq 3$, all partitions $\lambda\vdash n$ are hooks with $\len(\lambda)\leq 3$, therefore $I_\lambda$ is radical by Theorems~\ref{thm:Ia1} and~\ref{thm:a11}, that is $I_\lambda=I(V_\lambda)$. By Corollary~\ref{cor:posetvar}, $\lambda\mapsto V_\lambda$ from $(\partn,\trianglelefteq)$ to $(\{V_\lambda\mid\lambda\in \partn\},\supseteq)$ is an isomorphism of posets, and since $I_\lambda=I(V_\lambda)$, $\lambda\mapsto I_\lambda$ from $(\partn,\trianglelefteq)$ to $(\{I_\lambda\mid\lambda\in \partn\},\subseteq)$ is also an isomorphism of posets. Assume $m\geq 2$, $n>3$, and consider $(n-2,1,1)\trianglelefteq(n-2,2)$. By Proposition~\ref{prop:degreegen}, both ideals have minimum degree $D=2$. Let $f\in I_{(n-2,1,1)}$, $\deg(f)=2$. By way of contradiction, assume $\lambda\mapsto I_\lambda$ is an isomorphism of posets, then $I_{(n-2,1,1)}\subseteq I_{(n-2,2)}$, and
$f\in I_{(n-2,2)}$. 
By Remark~\ref{rmk:IiffR}, since $f$ has minimum degree in $I_{(n-2,1,1)}$, $f\in \mathcal{R}_{m,n,(n-2,1,1)}$, and since $f$ has minimum degree in $I_{(n-2,2)}$, $f\in \mathcal{R}_{m,n,(n-2,2)}$.
But $\mathcal{R}_{m,n,(n-2,1,1)}\cap\mathcal{R}_{m,n,(n-2,2)}=\{0\}$, which is a contradiction.
\end{proof}
\begin{remark}\label{rmk:inclusionradicality}
For all $m$ and $n$, the map $\lambda\mapsto I_\lambda$ is injective and reflects the order, since $I_\mu\subseteq I_\lambda$ implies $V_\lambda\subseteq V_\mu$, hence $\mu\trianglelefteq\lambda$ by Corollary~\ref{cor:posetvar}. What fails for $m\geq2$ and $n\geq4$ is the other direction, and only where $I_\lambda$ is not radical. Indeed, if $\mu\trianglelefteq\lambda$, then $V_\lambda\subseteq V_\mu$ and $I(V_\mu)\subseteq I(V_\lambda)$ by Corollary~\ref{cor:posetvar}, so if $I_\lambda$ is radical, $I_\mu\subseteq I(V_\mu)\subseteq I(V_\lambda)=I_\lambda$.  We will use this in Section~\ref{sec:nonradicalitycriteria}.
\end{remark}

\subsubsection{Non-radicality criteria}\label{sec:nonradicalitycriteria}
Let $p=\prod_{i=1}^n\mathbf{x}_{\bullet,i}^{\balpha_i}\in\mathcal{R}_{m,n}$
be a monomial with $i$-multiexponents $\balpha_i\in\N^m$. The content of
$p$ is the multiset of its multiexponents,
$c(p):=\{\{\balpha_1,\dots,\balpha_n\}\}$, the multidegree of $p$ is
their sum in $\N^m$, $\mdeg(p):=\sum_{i=1}^n\balpha_i\in \N^m$, and the degree of
$p$ is the sum of their total degrees,
$\deg(p):=\sum_{i=1}^n|\balpha_i|=|\mdeg(p)|\in\N$. In general, a content is a
multiset $\gamma=\{\{\bgamma_1,\dots,\bgamma_n\}\}$ of $n$ elements of
$\N^m$, its multidegree is $\mdeg(\gamma):=\sum_{i=1}^n\bgamma_i\in\N^m$, and
its degree is $\deg(\gamma):=|\mdeg(\gamma)|\in\N$. The action of $S_n$ preserves content, multidegree and degree.
We denote by $(\mathcal{R}_{m,n})_{\gamma}$ the subspace of polynomials
spanned by monomials of content $\gamma$, by
$(\mathcal{R}_{m,n})_{\balpha}$ the subspace of multihomogeneous
polynomials of multidegree $\balpha$ and by
$(\mathcal{R}_{m,n})_d$ the subspace of homogeneous
polynomials of degree $d$. 

We will give three criteria to establish non-radicality, using contents,
multidegrees and total degrees, respectively. They are presented from
most general to least general, but also from hardest to easiest to apply
to a given $(m,n,\lambda)$.

Since the action of $S_n$ preserves content, multidegree and degree, the subspaces $(\mathcal{R}_{m,n})_{\gamma}$, $(\mathcal{R}_{m,n})_{\balpha}$ and $(\mathcal{R}_{m,n})_d$ are $S_n$-submodules of $\mathcal{R}_{m,n}$, and the projection $\pi_\lambda$ maps each of them into itself. Therefore
\begin{equation*}
(\mathcal{R}_{m,n})_d=\bigoplus_{|\balpha|=d}(\mathcal{R}_{m,n})_{\balpha}
=\bigoplus_{|\balpha|=d}\left(\bigoplus_{\mdeg(\gamma)=\balpha}(\mathcal{R}_{m,n})_{\gamma}\right),
\end{equation*}
and for a subspace $M\subseteq\mathcal{R}_{m,n}$ we write
$M_{\gamma}:=M\cap(\mathcal{R}_{m,n})_{\gamma}$,
$M_{\balpha}:=M\cap(\mathcal{R}_{m,n})_{\balpha}$ and 
$M_d=M\cap(\mathcal{R}_{m,n})_d$.

We say that a content $\gamma$ is $\lambda$-compatible if
$(\mathcal{R}_{m,n,\lambda})_{\gamma}\neq0$, and we denote the set of
$\lambda$-compatible contents by $\Cont(\lambda)$.
Let $\gamma=\{\{\bgamma_1,\dots,\bgamma_n\}\}$ and
$\delta=\{\{\bdelta_1,\dots,\bdelta_n\}\}$ be contents. We say
$\delta<\gamma$ if $\delta\neq\gamma$ and there exists
$\sigma\in S_n$ such that $\bdelta_i\leq\bgamma_{\sigma(i)}$
for all $1\leq i\leq n$. 

Analogously, we say that a multidegree $\balpha$ is $\lambda$-compatible
if $(\mathcal{R}_{m,n,\lambda})_{\balpha}\neq0$, and we denote the set of
$\lambda$-compatible multidegrees by $\Mdeg(\lambda)$.
Let $\balpha,\bbeta\in\N^m$ be multidegrees. From now on, we say $\bbeta<\balpha$ if
$\bbeta\neq\balpha$ and $\beta_r\leq\alpha_r$ for all $1\leq r\leq m$. 

Note that if $\gamma\in\Cont(\lambda)$, then $\mdeg(\gamma)\in \Mdeg(\lambda)$, and if $\delta<\gamma$, then
$\mdeg(\delta)<\mdeg(\gamma)$.

The key observation is that multiplication by a monomial can only increase the content.

\begin{lemma}\label{lem:contentpiece}
Let $\lambda\vdash n$ and let $\gamma$ be a content such that
$\delta\notin\Cont(\lambda)$ for all $\delta<\gamma$. Then
\begin{equation*}
(I_\lambda)_{\gamma}=(\mathcal{R}_{m,n,\lambda})_{\gamma}.
\end{equation*}
\end{lemma}
\begin{proof}
Since multiplication by a monomial
can only increase content, taking content-$\gamma$ components gives
\begin{equation*}
(I_\lambda)_{\gamma}\subseteq(\mathcal{R}_{m,n,\lambda})_{\gamma}
+\sum_{\substack{\delta\in\Cont(\lambda)\\ \delta<\gamma}}
\left((\mathcal{R}_{m,n})(\mathcal{R}_{m,n,\lambda})_{\delta}\right)_{\gamma}.
\end{equation*}
Since $(\mathcal{R}_{m,n,\lambda})_{\delta}=0$ for all $\delta<\gamma$,
$(I_\lambda)_{\gamma}\subseteq(\mathcal{R}_{m,n,\lambda})_{\gamma}$, and $(\mathcal{R}_{m,n,\lambda})_{\gamma}\subseteq(I_\lambda)_{\gamma}$ by definition.
\end{proof}

\begin{prop}\label{prop:contentcriterion}
Let $\lambda\vdash n$. If there exists $\gamma\in\Cont(\mu)$,
$\mu\triangleleft\lambda$, such that $\delta\notin\Cont(\lambda)$
for all $\delta<\gamma$, then $I_\lambda\subseteq\mathcal{R}_{m,n}$ is
not radical.
\end{prop}
\begin{proof}
Since $\gamma\in\Cont(\mu)$, there exists a nonzero
$f\in(\mathcal{R}_{m,n,\mu})_{\gamma}$. Then
$f\in I_\mu\subseteq I(V_\mu)$, and by Corollary~\ref{cor:posetvar} (see
Remark~\ref{rmk:inclusionradicality}), $I(V_\mu)\subseteq I(V_\lambda)$, therefore
$f\in I(V_\lambda)$. On the other hand, if $f\in I_\lambda$, then
$f\in(I_\lambda)_{\gamma}=(\mathcal{R}_{m,n,\lambda})_{\gamma}$ by
Lemma~\ref{lem:contentpiece}, which contradicts
$\mathcal{R}_{m,n,\mu}\cap\mathcal{R}_{m,n,\lambda}=\{0\}$. Then
$f\in I(V_\lambda)\setminus I_\lambda$, and $I_\lambda$ is not radical.
\end{proof}

\begin{corollary}\label{cor:mdegcriterion}
Let $\lambda\vdash n$. If there exists $\balpha\in \Mdeg(\mu)$,
$\mu\triangleleft\lambda$, such that $\bbeta\notin \Mdeg(\lambda)$ for all
$\bbeta<\balpha$, then $I_\lambda\subseteq\mathcal{R}_{m,n}$ is not
radical.
\end{corollary}
\begin{proof}
Since $\balpha\in \Mdeg(\mu)$ and
$(\mathcal{R}_{m,n,\mu})_{\balpha}
=\bigoplus_{\mdeg(\gamma)=\balpha}(\mathcal{R}_{m,n,\mu})_{\gamma}$,
there exists $\gamma\in\Cont(\mu)$ with
$\mdeg(\gamma)=\balpha$. For any $\delta<\gamma$, we have
$\mdeg(\delta)<\mdeg(\gamma)=\balpha$, so
$\mdeg(\delta)\notin \Mdeg(\lambda)$, therefore
$\delta\notin\Cont(\lambda)$ and by
Proposition~\ref{prop:contentcriterion}, $I_\lambda$ is not radical.
\end{proof}

\begin{corollary}\label{cor:totaldegreecriterion}
Let $\lambda\vdash n$. If there exists $\mu\triangleleft\lambda$ with
$D_\mu=D_\lambda$, then $I_\lambda\subseteq\mathcal{R}_{m,n}$ is not
radical.
\end{corollary}
\begin{proof}
Since $D_\mu$ is the minimum degree of $\mathcal{R}_{m,n,\mu}$, there
exists $\balpha\in \Mdeg(\mu)$ with $|\balpha|=D_\mu$. For any
$\bbeta<\balpha$, we have $|\bbeta|<|\balpha|=D_\lambda$, so
$\bbeta\notin \Mdeg(\lambda)$, since $D_\lambda$ is the minimum degree of
$\mathcal{R}_{m,n,\lambda}$, and by 
Corollary~\ref{cor:mdegcriterion}, $I_\lambda$ is not radical.
\end{proof}
\begin{thm}\label{thm:nonradical}
Let $\lambda=(\lambda_1,\dots,\lambda_\ell)\vdash n$, and assume there
exist a band index $b$ and rows $i<j$ with $N_{b-1}<i<j\leq \min(N_b,\ell+1)$ and
$\lambda_j+1<\lambda_i$. Then $I_\lambda\subseteq\mathcal{R}_{m,n}$ is not
radical.
\end{thm}
\begin{proof}
    Let $I$ be the last row with $\lambda_I=\lambda_i$ and $J$ the first row
with $\lambda_J=\lambda_j$. Since $\lambda$ is weakly decreasing and
$\lambda_J<\lambda_I$, we have $i\leq I<J\leq j$, and therefore
$N_{b-1}<I<J\leq N_b$.
Since $I$ is
the last row with value $\lambda_i$, $\lambda_{I+1}<\lambda_I$, and since $J$ is the
first row with its value, $\lambda_{J-1}>\lambda_J$. We have $\lambda_J+1\leq\lambda_I-1$ by hypothesis, so
\begin{equation*}
\mu:=(\lambda_1,\dots,\lambda_I-1,\dots,\lambda_J+1,\dots,\lambda_\ell)
\end{equation*}
is a partition (with the convention $\lambda_{\ell+1}=0$, so that if
$J=\ell+1$, $\mu=(\lambda_1,\dots,\lambda_I-1,\dots,\lambda_\ell,1)$),
obtained from $\lambda$ by moving one box from row $I$ to row
$J$, and $\mu\triangleleft\lambda$. By Proposition~\ref{prop:degreegen}, the minimum degree depends on the number of boxes in each band of the partition. Since $I$ and $J$ are in the same band $(N_{b-1},N_b]$, moving one box from row $I$ to row $J$ does not change minimum degree, $D_\mu=D_\lambda$. By
Corollary~\ref{cor:totaldegreecriterion}, $I_\lambda$ is not radical.
\end{proof}
\begin{corollary}\label{cor:nonhooks}
Let $\lambda=(\lambda_1,\dots,\lambda_\ell)\vdash n$, $\lambda_2>1$, $m\geq\ell$, then $I_\lambda\subseteq\mathcal{R}_{m,n}$ is not radical. In particular, for a fixed non-hook partition $\lambda$, $I_\lambda\subseteq\mathcal{R}_{m,n}$ is not radical for all but a finite number of $m$.
\end{corollary}
\begin{proof}
If $m\geq \ell$, $1=N_0<\ell< N_1=1+m$. Take $i=2$, $j=\ell+1$. Then  Theorem~\ref{thm:nonradical} applies with $b=1$.
\end{proof}
\subsubsection{Examples}
The following examples illustrate the use of the multidegree criteria (Corollary~\ref{cor:mdegcriterion}) and total degree criteria (Corollary~\ref{cor:totaldegreecriterion}). All the examples are for $m=2$.
\begin{example}
Let $\lambda=(a,b)$ with $b\geq2$ and $\mu=(a,b-1,1)\triangleleft\lambda$. We have $D_\lambda=D_\mu=b$, and Corollary~\ref{cor:totaldegreecriterion} shows that $I_{(a,b)}\subseteq\mathcal{R}_{2,n}$ is not radical. Explicitly, if $A$ is the superstandard tableau of shape $\mu$ and $E$ its minimal filling, then $f=\varepsilon_A(\mathbf{x}_A^E)$ is a nonzero element of $I(V_{(a,b)})\setminus I_{(a,b)}$.
\end{example}
\begin{example}\label{ex:222}
Let $\lambda=(2,2,2)$. Note that for any $\mu\triangleleft \lambda$, $D_\mu>D_\lambda$, and we cannot use Corollary~\ref{cor:totaldegreecriterion}. We have $D_\lambda=4$. A minimal degree filling of $\lambda$ has $(0,0),(0,1),(1,0)$ in each of the two columns, so the unique
multidegree in $\Mdeg(\lambda)$ of total degree $4$ is $(2,2)$. Let $\mu=(2,2,1,1)$, and let $E$ be the minimal filling of $\mu$, then $\mdeg(E)=(1,4)\in \Mdeg(\mu)$. Any
$\bbeta<(1,4)$ has total degree at most $4$, and the only element of
$\Mdeg(\lambda)$ of total degree at most $4$ is $(2,2)\not<(1,4)$, so
$\bbeta\notin \Mdeg(\lambda)$. By Corollary~\ref{cor:mdegcriterion},
$I_{(2,2,2)}\subseteq\mathcal{R}_{2,6}$ is not radical.
\end{example}
\begin{example}\label{ex:abc}
More generally, let $\lambda=(a,b,c)$ with $c\geq2$. Then
$D_\lambda=b+c$, realized exactly in the multidegrees $(c+i,\,b-i)$
for $0\leq i\leq b-c$. Let $\mu=(a,b,c-1,1)\triangleleft\lambda$, and let $E$ be the minimal filling of $\mu$. Then $\balpha:=\mdeg(E)=(c-1,\,b+2)\in \Mdeg(\mu)$, and since $c+i>c-1$ for all
$i\geq0$, no element of $\Mdeg(\lambda)$ is $<\balpha$. As in
Example~\ref{ex:222}, Corollary~\ref{cor:mdegcriterion} shows that
$I_{(a,b,c)}\subseteq\mathcal{R}_{2,n}$ is not radical.
\end{example}
\begin{example}
Corollary~\ref{cor:mdegcriterion} also applies when Theorem~\ref{thm:nonradical}
does not. Let $\lambda=(3,3,3,3,3,2,1)\vdash18$. Then $D_\lambda=25$, realized
exactly in the multidegrees $(10+i,15-i)$, $0\leq i\leq5$, while for
$\mu=(3,3,3,3,2,2,2)\triangleleft\lambda$ the minimal filling gives
$\balpha=(9,17)\in \Mdeg(\mu)$. Any $\bbeta<\balpha$ has $\beta_1\leq9$ and total
degree at most $25$, so $\bbeta\notin \Mdeg(\lambda)$, and
$I_\lambda\subseteq\mathcal{R}_{2,18}$ is not radical.
\end{example}
\begin{example}
The criteria do not always apply. For $\lambda=(a,2,1)$ with $a \geqslant 2$, one has
$D_\mu>D_\lambda$ for every $\mu\triangleleft\lambda$, so
Corollary~\ref{cor:totaldegreecriterion} cannot be used. For
$\lambda=(2,2,1,1)$, Corollary~\ref{cor:mdegcriterion} fails as well: the
elements of $\Mdeg(\lambda)$ of degree $D_\lambda=5$ are $(1,4)$, $(2,3)$, $(3,2)$,
$(4,1)$, while for $\mu=(2,1,1,1,1)$ the fillings of minimal degree have
multidegrees $(2,4)$, $(3,3)$, $(4,2)$, each of which dominates one of them.
Whether $I_{(2,2,1,1)}\subseteq\mathcal{R}_{2,6}$ is radical is open (see
Table~\ref{tab:threshold}).
\end{example}
\subsection{Summary}\label{subsec:summary}
Note that Theorem~\ref{thm:nonradical} is stronger than Corollary~\ref{cor:nonhooks}. For example, let $\lambda=(3,3,3,3,2,1)\vdash15$, so $\len(\lambda)=6$. For $m=2$, $N_0=1$, $N_1=3$, $N_2=6$. Taking $b=2$, $i=4$, $j=6$, both in the band $(N_1,N_2]=(3,6]$, we have $\lambda_j+1=\lambda_6+1=2<3=\lambda_4=\lambda_i$, so Theorem~\ref{thm:nonradical} shows that $I_\lambda\subseteq\mathcal{R}_{2,15}$ is not radical, even though $m=2<6=\len(\lambda)$.

Corollary~\ref{cor:nonhooks} shows that for any non-hook partition $\lambda$,
the ideal $I_\lambda\subseteq\mathcal{R}_{m,n}$ is not radical
for all $m\geq\ell$,
and this is the best we can do in the sense that for $m<\ell$, 
the ideal $I_\lambda\subseteq\mathcal{R}_{m,n}$ may be radical. Note that by 
Corollary~\ref{cor:closure}, $\rad(I_\lambda)$ can be computed explicitly as
\begin{equation*}
\rad(I_\lambda)=I(V_\lambda)=
I\left(\bigcup_{\mu\in\mathcal{A}_\lambda}\overline{H_\mu}
\right)
=\bigcap_{\mu\in\mathcal{A}_\lambda}I\left(\overline{H_\mu}\right),\quad \mathcal{A}_\lambda=\{\mu\not\trianglelefteq\lambda,\,\mu\text{ minimal by }\prec\}.
\end{equation*}
Consider, for example, $I_{(2,2,1)}\subseteq\mathcal{R}_{2,5}$. We can use \texttt{SageMath} to compute $\rad(I_{(2,2,1)})=I(\overline{H_{(3,1,1)}})$ and verify that for some $d$, 
\begin{equation*}
\rad(I_{(2,2,1)})\subseteq\langle\varepsilon_T(p),\,p\in\Mon(\mathcal{R}_{2,5}),\,\deg(p)\leq d,\,T\in\SYT((2,2,1))\rangle. 
\end{equation*}
We conclude that $I_{(2,2,1)}\subseteq\mathcal{R}_{2,5}$ is radical. The same computation shows that $I_{(3,2,1)}\subseteq\mathcal{R}_{2,6}$,  $I_{(1^4)}\subseteq\mathcal{R}_{4,4}$, and $I_{(2,1^3)}\subseteq\mathcal{R}_{5,5}$ are radical.

The characterization 
of all $(m,n,\lambda)$ for which $I_\lambda\subseteq\mathcal{R}_{m,n}$ 
is radical remains open.

Table~\ref{tab:threshold} summarizes the radicality results for $n\leq7$.

\begin{table}[H]
\centering
\footnotesize
\begin{tabular}{@{}l cccc c l@{}}
\toprule
$\lambda$ & $m=2$ & $m=3$ & $m=4$ & $m=5$ & $m(\lambda)$ & Reference \\
\midrule
\multicolumn{7}{@{}l}{\emph{Hook partitions, any $n$}}\\
$(n)$              & \checkmark & \checkmark & \checkmark & \checkmark & $\infty$ & $I_{(n)}=\mathcal{R}_{m,n}$ \\
$(n-1,1)$          & \checkmark & \checkmark & \checkmark & \checkmark & $\infty$ & Theorem~\ref{thm:Ia1} \\
$(n-2,1,1)$        & \checkmark & \checkmark & \checkmark & \checkmark & $\infty$ & Theorem~\ref{thm:a11} \\
$(1^4)$            & \checkmark & \checkmark & \checkmark & \checkmark & $\infty$ & computations in $\mathcal{R}_{4,4}$ and Theorem~\ref{thm:radicalinRnn} \\

$(2,1^3)$            & \checkmark & \checkmark & \checkmark & \checkmark & $\infty$ & computations in $\mathcal{R}_{5,5}$ and Theorem~\ref{thm:radicalinRnn} \\
$(3,1^3)$            & \checkmark & \checkmark & ? & ? & $\geq4$ & computations in $\mathcal{R}_{3,6}$\\
$(4,1^3)$            & \checkmark & ? & ? & ? & $\geq3$ & computations in $\mathcal{R}_{2,7}$ \\
$(1^n)$, $n\geq5$  & \checkmark & ?          & ?          & ?          & $>2$     & Remark~\ref{rmk:haiman} \\
remaining hooks    & ?          & ?          & ?          & ?          & open     & Conjecture~\ref{conj:hook} \\
\midrule
\multicolumn{7}{@{}l}{\emph{Non-hook partitions, $n=4$}}\\
$(2,2)$            & $\times$ & $\times$ & $\times$ & $\times$ & $2$ & Corollary~\ref{cor:nonhooks} \\
\midrule
\multicolumn{7}{@{}l}{\emph{Non-hook partitions, $n=5$}}\\
$(3,2)$            & $\times$   & $\times$ & $\times$ & $\times$ & $2$ & Corollary~\ref{cor:nonhooks} \\
$(2,2,1)$          & \checkmark & $\times$ & $\times$ & $\times$ & $3$ & \S\ref{subsec:summary}; Corollary~\ref{cor:nonhooks} \\
\midrule
\multicolumn{7}{@{}l}{\emph{Non-hook partitions, $n=6$}}\\
$(4,2)$, $(3,3)$   & $\times$   & $\times$ & $\times$ & $\times$ & $2$      & Corollary~\ref{cor:nonhooks} \\
$(3,2,1)$          & \checkmark & $\times$ & $\times$ & $\times$ & $3$      & \S\ref{subsec:summary}; Corollary~\ref{cor:nonhooks} \\
$(2,2,2)$          & $\times$   & $\times$ & $\times$ & $\times$ & $2$      & Example~\ref{ex:222} \\
$(2,2,1,1)$        & ?          & ?        & $\times$ & $\times$ & $\leq4$  & Corollary~\ref{cor:nonhooks} \\
\midrule
\multicolumn{7}{@{}l}{\emph{Non-hook partitions, $n=7$}}\\
$(5,2)$, $(4,3)$   & $\times$ & $\times$ & $\times$ & $\times$ & $2$      & Corollary~\ref{cor:nonhooks} \\
$(4,2,1)$          & ?        & $\times$ & $\times$ & $\times$ & $\leq3$  & Corollary~\ref{cor:nonhooks} \\
$(3,3,1)$          & $\times$ & $\times$ & $\times$ & $\times$ & $2$      & Theorem~\ref{thm:nonradical} \\
$(3,2,2)$          & $\times$ & $\times$ & $\times$ & $\times$ & $2$      & Example~\ref{ex:abc} \\
$(3,2,1,1)$, $(2,2,2,1)$ & ? & ?        & $\times$ & $\times$ & $\leq4$  & Corollary~\ref{cor:nonhooks} \\
$(2,2,1,1,1)$      & ?        & ?        & ?        & $\times$ & $\leq5$  & Corollary~\ref{cor:nonhooks} \\
\bottomrule
\end{tabular}
\caption{Radicality of $I_\lambda\subseteq\mathcal{R}_{m,n}$ for $n\leq7$ and $2\leq m\leq5$: \checkmark\ radical, $\times$ not radical, ? open. For $m=1$ every Specht ideal is radical~\cite{Murai2022}, and for $n\leq3$ every partition is a hook of length at most three. By Proposition~\ref{prop:bound}, a $\times$ propagates to larger $m$ and a \checkmark\ to smaller $m$; the rows marked $\infty$ are radical for every $m$. The last column names the result establishing the first $\times$, and the computation establishing a \checkmark\ that is not covered by a theorem. All computations are done using \texttt{SageMath} 10.6~\cite{sagemath} and available at \url{https://github.com/uggtos/diagonal_specht}.}
\label{tab:threshold}
\end{table}

\section{Cohen--Macaulayness}
\label{sec:cohen-macaulayness}
For $m=1$ and in characteristic $0$, Yanagawa~\cite{Yanagawa2021} showed that $\mathcal{R}_{1,n}/I_\lambda$ is Cohen--Macaulay if and only if $\lambda$ is of the form $(a,1^{n-a})$, $(a,b)$ or $(a,a,1)$. In this section we prove that for $m\geq2$, $\mathcal{R}_{m,n}/I_\lambda$ and $\mathcal{R}_{m,n}/\rad(I_\lambda)$ are Cohen--Macaulay if and only if $\lambda=(n)$ or $\lambda=(n-1,1)$ (Theorem~\ref{thm:CM}). The argument only uses the minimal primes of the ideal, which are the same for $I_\lambda$ and $\rad(I_\lambda)$, therefore the proof applies to both. The key tool is the following result.
\begin{thm}[{\cite[Proposition~2.1]{Hartshorne1962}}]\label{thm:hartshorne}
Let $(R,\mathfrak{m})$ be a Noetherian local ring with $\depth(R)\geq2$. Then the punctured spectrum $\Spec(R)\setminus\{\mathfrak{m}\}$ is connected.
In particular, if $R$ is Cohen--Macaulay and $\dim(R)\geq 2$, then $\Spec(R)\setminus\{\mathfrak{m}\}$ is connected.
\end{thm}
We will also use the following standard commutative algebra results.
\begin{prop}[{\cite[Proposition~1.1, Proposition~3.11 (iv)]{atiyamacdonald}}]\label{prop:standard}
Let $R$ be a ring.
    \begin{enumerate}
        \item Let $\mathfrak{a}$ be an ideal of $R$, and $\varphi:R\longrightarrow R/\mathfrak{a}$ be the projection map. Then $\mathfrak{b}\mapsto\varphi(\mathfrak{b})$ is a one-to-one order-preserving correspondence between ideals $\mathfrak{b}$ of $R$ with $\mathfrak{a}\subseteq\mathfrak{b}$ and ideals of $R/\mathfrak{a}$.
        \item Let $\mathfrak{q}$ be a prime ideal of $R$, and let $R_{\mathfrak{q}}$ be the localization of $R$ at $\mathfrak{q}$. Then $\mathfrak{p}\mapsto\mathfrak{p}R_{\mathfrak{q}}$ is a one-to-one order-preserving correspondence between prime ideals $\mathfrak{p}$ of $R$ with $\mathfrak{p}\subseteq\mathfrak{q}$ and prime ideals of $R_{\mathfrak{q}}$. 
    \end{enumerate}
\end{prop}

Given a set partition $\Pi$, its \emph{support} $\Supp(\Pi)$ is the union of all subsets $\Pi_l\in\Pi$ such that $|\Pi_l|>1$.

Let $\lambda\vdash n$, $\lambda_1\leq n-2$, and let $I$ denote either $I_\lambda$ or $\rad(I_\lambda)$. Let $\mathcal{S}:=\mathcal{R}_{m,n}/I$, and let $\varphi:\mathcal{R}_{m,n}\to\mathcal{S}$ be the projection. By Corollary~\ref{cor:linearspaces}, both ideals have the same minimal primes, $\mathfrak{p}_\Pi$ with $\sh(\Pi)\in\mathcal{A}_\lambda$. Let $k:=\lambda_1+1\leq n-1$ and consider the set partitions
\begin{equation*}
    \Theta:=\{\{1,\dots,k+1\},\{k+2\},\dots,\{n\}\}\quad\text{and}\quad 
    \Pi^\sharp:=\{\{1,\dots,k\},\{k+1\},\dots,\{n\}\},
\end{equation*}
note that $\sh(\Pi^\sharp)=\lambda^\sharp$, and let $\fP:=\mathfrak{p}_{\Pi^\sharp}$, which is a minimal prime by Corollary~\ref{cor:linearspaces} and Lemma~\ref{lem:lambdasharp}, and $\fQ:=\mathfrak{p}_\Theta$. We have
\begin{equation*}
    I_\lambda \subseteq \rad(I_\lambda) \subseteq \fP \subseteq \fQ.
\end{equation*}
By Proposition~\ref{prop:standard}(1), $\overline{\fP}:=\varphi(\fP)$ and $\overline{\fQ}:=\varphi(\fQ)$ are prime ideals of $\mathcal{S}$ with $\overline{\fP}\subsetneq\overline{\fQ}$. Consider the local ring $\mathcal{T}=\mathcal{S}_{\overline{\fQ}}$, whose maximal ideal is $\mathfrak{m}=\overline{\fQ}\mathcal{S}_{\overline{\fQ}}$, and let $X:=\Spec(\mathcal{S}_{\overline{\fQ}})$. Let $\fP'=\overline{\fP}\mathcal{T}$.

\begin{lemma}\label{lem:specdisc}
    With the notation above. If $\lambda_1<n-1$ then $X\setminus\{\mathfrak{m}\}$ is disconnected.
\end{lemma}

\begin{proof}
    By Corollary~\ref{cor:linearspaces}, the minimal primes of $I$ are the $\mathfrak{p}_\Pi$ with $\sh(\Pi)\in\mathcal{A}_\lambda$. By Proposition~\ref{prop:standard}(1), the minimal primes of $\mathcal{S}$ are their images $\varphi(\mathfrak{p}_\Pi)$, and by Proposition~\ref{prop:standard}(2), the minimal primes of $\mathcal{T}$ are the ideals $\mathfrak{T}_\Pi:=\varphi(\mathfrak{p}_\Pi)\mathcal{T}$ for which $\varphi(\mathfrak{p}_\Pi)\subseteq\overline{\fQ}$, that is, with $\mathfrak{p}_\Pi\subseteq\fQ$.

We have $\mathfrak{p}_\Pi\subseteq\fQ$ if and only if $W_\Theta\subseteq W_\Pi$, if and only if $\Pi\prec\Theta$, if and only if $\Supp(\Pi)\subseteq\{1,\dots,k+1\}$, since $\{1,\dots,k+1\}$ is the only subset of $\Theta$ with more than one element.

Define
    \begin{equation*}
        X_1:=V(\fP')=\{\textrm{ prime ideals of }\mathcal{T}\textrm{ containing }\fP'\},\quad X_2:=\bigcup_{\substack{\mathfrak{T} \in \{\textrm{minimal primes of }\mathcal{T}\} \\ \mathfrak{T} \neq \fP'}}V(\mathfrak T).
    \end{equation*}
We will show that both $X_1$ and $X_2$ are closed and nonempty, and $X=X_1\cup X_2$, $X_1\cap X_2=\{\mathfrak{m}\}$. By definition $X_1$ is closed and $X_2$ is closed because it is a finite union of closed sets. Since $X_1\cup X_2$ is the set of all primes that contain one of the minimal primes,  $X=X_1 \cup X_2$. We have $\fP'\in X_1\setminus\{\mathfrak{m}\}$, as $\fP\subsetneq\fQ$. For $\Pi':=\{\{1\},\{2,\dots,k+1\},\{k+2\},\dots,\{n\}\}$ we have $\sh(\Pi')=\lambda^\sharp\in\mathcal{A}_\lambda$, $\Supp(\Pi')\subseteq\{1,\dots,k+1\}$ and $\Pi'\neq\Pi^\sharp$, so $\mathfrak{T}_{\Pi'}$ is a minimal prime of $\mathcal{T}$ different from $\fP'$, and $\mathfrak{T}_{\Pi'}\in X_2\setminus\{\mathfrak{m}\}$, as $\mathfrak{p}_{\Pi'}\subsetneq\fQ$. Lastly, we want to show that $X_1 \cap X_2= \{\mathfrak{m}\}.$ Let $\mathfrak{n}\in X_1\cap X_2$, then $\mathfrak{n}$ contains all the images of the generators of $\fP$. It also contains the images of the generators of some minimal prime ideal $\mathfrak{T} \neq \fP'$ of $\mathcal{T}$ occurring in the definition of $X_2$.
This minimal prime ideal corresponds to a set partition $\Pi$ of $[n]$, whose support is included in $\{1,\cdots,k+1\}$. If its support was included in $\{1,\cdots,k\}$, then its shape $\nu$ would refine $\lambda^\sharp=(1+\lambda_1,1,\cdots,1) \in \mathcal{A}_\lambda$, which means that $\nu = \lambda^\sharp $, and in turn $\mathfrak{T} = \fP'$. So, there exists $j \in \{1,\cdots,k\}$ such that $j \sim_\Pi k+1$. Then, for every $1\leq r\leq m$, the images in $\mathcal{T}$ of $x_{r,j}-x_{r,k+1}$ and $x_{r,1}-x_{r,j}$ lie in $\mathfrak{T}\subseteq\mathfrak{n}$ and in $\fP'\subseteq\mathfrak{n}$ respectively, so the image of $x_{r,1}-x_{r,k+1}$ lies in $\mathfrak{n}$. Hence $\mathfrak{n}$ contains the images of $x_{r,1}-x_{r,a}$ for all $1\leq r\leq m$ and $2\leq a\leq k+1$, which generate $\mathfrak{m}$, so $\mathfrak{n}=\mathfrak{m}$. Therefore $X\setminus\{\mathfrak{m}\}=(X_1\setminus\{\mathfrak{m}\})\sqcup(X_2\setminus\{\mathfrak{m}\})$ is a union of two disjoint nonempty closed subsets, hence disconnected.
 \end{proof}

\begin{thm}\label{thm:CM}
Let $\lambda\vdash n$ and $m\geq2$. Then $\mathcal{R}_{m,n}/I_\lambda$ is
Cohen--Macaulay if and only if $\lambda=(n)$ or $\lambda=(n-1,1)$.
The same holds for $\mathcal{R}_{m,n}/\rad(I_\lambda)$.
\end{thm}

\begin{proof}
For $\lambda=(n)$ we have $I_\lambda=\mathcal{R}_{m,n}$, $\mathcal{R}_{m,n}/I_\lambda=\mathcal{R}_{m,n}/\rad(I_\lambda)=0$, which is Cohen--Macaulay by convention. For $\lambda=(n-1,1)$, by Theorem~\ref{thm:Ia1}, $I_\lambda=\rad(I_\lambda)=\langle x_{r,i}-x_{r,1}\mid 1\leq r\leq m,\ 2\leq i\leq n\rangle$, and $\mathcal{R}_{m,n}/I_\lambda=\mathcal{R}_{m,n}/\rad(I_\lambda)\cong \C[x_{1,1},\dots,x_{m,1}]$, $x_{r,i}\mapsto x_{r,1}$, $\mathcal{R}_{m,n}/I$ is a polynomial ring, therefore it is Cohen--Macaulay.

Let $\lambda_1\leq n-2$, let $I$, $\mathcal{S}$, $\mathcal{T}$, $\fP$, $\fQ$, $\mathfrak{m}$ be as before. Suppose that $\mathcal{S}$ is Cohen--Macaulay. Then so is $\mathcal{T}=\mathcal{S}_{\overline{\fQ}}$. We want to show that $\dim(\mathcal{T})\geq2$. Since $\fQ/\fP$ is generated by the $m$ independent linear forms $x_{r,k+1}-x_{r,1}$ in $\mathcal{R}_{m,n}/\fP$, we have $\dim\mathcal{T}\geq\operatorname{ht}(\fQ/\fP)=m\geq2$.

Since $\mathcal{T}$ is a Cohen--Macaulay local ring of dimension at least $2$, Theorem~\ref{thm:hartshorne} shows that $\Spec(\mathcal{T})\setminus\{\mathfrak{m}\}$ is connected, contradicting Lemma~\ref{lem:specdisc}. Hence $\mathcal{S}$ is not Cohen--Macaulay.
\end{proof}

\section{Open questions}\label{sec:open}
We close with some questions suggested by the results above.

\subsubsection*{Radicality for hooks.} Is $I_\lambda\subseteq\mathcal{R}_{m,n}$ radical for every hook partition $\lambda=(a,1^{n-a})$ and every $m$? This holds for hooks of length at most three (Theorem~\ref{thmC}), and for $\lambda=(1^n)$ when $m=2$ (Remark~\ref{rmk:haiman}), but the general case remains open.

\subsubsection*{The radicality threshold.}
Determine the threshold $m(\lambda)$ of Definition~\ref{def:threshold} combinatorially. The bound $m(\lambda)\leq\len(\lambda)$ for non-hook $\lambda$ (Theorem~\ref{thmD}) is general but not always sharp, see Table~\ref{tab:threshold}.

\subsubsection*{Generators of the radical.}
The radical $\rad(I_\lambda)$ has an explicit description as an intersection of linear primes (Corollary~\ref{cor:linearspaces}). Is there a natural representation-theoretic generating set for $\rad(I_\lambda)$? More specifically, can one describe the $\GL_m\times S_n$-types that occur minimally in $\rad(I_\lambda)/I_\lambda$?

\subsubsection*{Refined $\GL_m\times S_n$-isotypic ideals.}
The $S_n$-isotypic component $\mathcal{R}_{m,n,\lambda}$ is already $\GL_m$-stable and decomposes into $\GL_m\times S_n$-isotypic summands indexed by pairs $(\nu,\lambda)$. For such a pair, one may study the ideal generated by the corresponding $(\nu,\lambda)$-isotypic component. It would be interesting to determine their varieties, radicals, and possible relations with diagonal collision arrangements.

\subsubsection*{Homological invariants.}
Beyond Cohen--Macaulayness, determine the depth, projective dimension, regularity, and Betti tables of $I_\lambda$ and $\rad(I_\lambda)$. Since the minimal primes are indexed by set partitions, these invariants may admit a combinatorial formulation.

\subsection*{Declaration of generative AI and AI-assisted technologies in the manuscript preparation process.}
$ $\\
\noindent Statement: During the preparation of this work the authors used OpenAI (chat.uit.no)  in order to improve grammar, spelling, and readability, as well as to improve the SageMath script. After using these tools, the authors reviewed and edited the content as needed and take full responsibility for the content of the published article.
\printbibliography
\end{document}